\documentclass[12pt, english, reqno]{amsart}
\usepackage[utf8]{inputenc}
\usepackage[OT1]{fontenc}

\usepackage{hyperref}
\usepackage{amssymb}
\usepackage{amsmath}
\usepackage{ bbold }
\usepackage{ dsfont }

\usepackage{graphics}
\usepackage{amsthm}
\newtheorem{theorem}{Theorem}[section]
\newtheorem{proposition}[theorem]{Proposition}

\newtheorem{lemma}[theorem]{Lemma}

\newtheorem{definition}[theorem]{Definition}
\newtheorem{example}[theorem]{Example}

\newtheorem{remark}[theorem]{Remark}

\newtheorem{innercustomthm}{}
\newenvironment{customthm}[1]
  {\renewcommand\theinnercustomthm{#1}\innercustomthm}
  {\endinnercustomthm}

\numberwithin{equation}{section}

\usepackage{tagging}
\newcommand{\settheoremtag}[1]{% \settheoremtag{<tag>}
	\let\oldthebigthm\thebigthm% Store \thetheorem
	\renewcommand{\thebigthm}{#1}% Redefine it to a fixed value
	\g@addto@macro\endbigset{% At \end{theorem}, ...
		\addtocounter{bigthm}{-1}% ...restore theorem counter value and...
		\global\let\thebigthm\oldthebigthm}% ...restore \thetheorem
}

\usepackage{mathrsfs}

\usepackage{amsmath}
\usepackage{esint}
\usepackage{mathtools}
\usepackage{amssymb}
\usepackage{multicol}
\usepackage{color}
\usepackage[makeroom]{cancel} 
\usepackage{multicol}
\usepackage{upgreek}
\usepackage[top=2cm, bottom=3cm, left=2cm, right=3cm]{geometry}
\usepackage{graphicx}
{}

\usepackage{makeidx}
\usepackage{tikz}
\usepackage{tikz-cd}
\usepackage{pgfplots}
\graphicspath{ {images/} }
\graphicspath{{/texfiles/}}
\usepackage[english]{babel}
\usepackage{hyphenat}
\pagecolor{white}
\usepackage{fancyhdr}
\usepackage{appendix}
\usepackage{ tipa }
\usepackage{listings}
\usepackage{color}
\usepackage{graphicx}

\renewcommand{\leq}{\leqslant}
\renewcommand{\geq}{\geqslant}
\renewcommand{\leq}{\leqslant}
\renewcommand{\geq}{\geqslant}

\hypersetup{
    colorlinks=true,
    linkcolor=blue,
    filecolor=magenta,      
    urlcolor=cyan,
    pdftitle={Singular CYG},
}

\title{Singular Gauduchon Conjecture}
\author{Guilherme Cerqueira-Gonçalves}
\address{Université de Toulouse, CNRS, Institut de Mathématiques de Toulouse, France}
\email{guicerqg@gmail.com / guilherme.cerqueira\mathunderscore{}goncalves@math.univ-toulouse.fr}
\urladdr{\href{https://guicerqmath.github.io/}{https://guicerqmath.github.io/}}
\date{December 17, 2025}

\pgfplotsset{compat=1.18}
\begin{document}

\begin{abstract}
In 1984 Gauduchon conjectured that one can find Gauduchon metrics
with prescribed Ricci curvature on all compact complex manifolds.
This conjecture was settled by Székelyhidi-Tosatti-Weinkove~\cite{TW17,TW19,STW17}
by the study of the Monge-Ampère equation for $(n-1)$-plurisubharmonic functions with a gradient term.
In this paper we study a singular version of this conjecture.
We obtain a $C^{0}$-estimate for this problem, without gradient terms, in smoothable hermitian variaties by adapting a technique of~\cite{GL25}.
We also prove the smoothness of solutions on holomorphic Kähler families, generalizing~\cite{TW17}.
\end{abstract}

\maketitle

\tableofcontents

\section*{Introduction}\label{CYG: background}

Yau's solution of the Calabi conjecture~\cite{Yau78}
prescribes the Ricci form of a Kähler metric in a fixed Kähler class.
In the Hermitian setting Gauduchon~\cite[IV.5]{Gau84} proposed the analogous problem for Gauduchon metrics, i.e. a Hermitian metric $\omega_{G}$ such that $\partial \overline{\partial}\omega_{G}^{n-1} = 0$.
This problem, the Gauduchon conjecture, was settled by Székelyhidi, Tosatti and Weinkove~\cite{TW17,TW19,STW17}. 
This article studies a singular version of it.

Gauduchon metrics are ubiquitous objects in Hermitian geometry since they exist in the conformal class of any Hermitian metric by~\cite{Gau77}.
%Li and Yau used them in~\cite{LY87} to define slope stability of vector bundles over compact non-Kähler manifolds and prove a Uhlenbeck-Yau Theorem in the Hermitian setting.
%Angella, Calamai and Spotti used them in~\cite{ACS17} to define the Gauduchon degree to study solvability of the Chern-Yamabe problem.
Recently, Barbaro and Otiman~\cite{BO25} investigated Calabi-Yau-Gauduchon metrics (Gauduchon + Chern-Ricci flat) on locally conformally Kähler manifolds.
The Gauduchon cone, related to the pseudo-effective cone~\cite{L99}, also motivates a finer study of such metrics,
see~\cite{Pop15,PU18,OV25a,OV25b}.

Degenerate PDEs on singular spaces are tightly linked to the Minimal Model Program and moduli theory.
Some recent developments are the Kähler singular version of Yau's theorem by Eyssidieux, Guedj and Zeriahi~\cite{EGZ09}
and the construction of Gauduchon metrics on singular Hermitian varieties by Pan~\cite{Pan22,Pan25}. 
An important open problem about singular Hermitian varieties is Reid's fantasy~\cite{Rei87} which conjectures
that Calabi-Yau threefolds form a connected moduli space by means of conifold transitions introduced by Clemens and Friedman~\cite{Cle83,Fri86}
consisting of smoothings and bimeromorphic transformations of non-Kähler Calabi-Yau varieties.
These constructions are believed to have intimate relationships with Hull-Strominger systems~\cite{Hul86,Str86}.
On such systems one needs to study a Calabi-Yau semi-Kähler metric (also known as balanced, defined as $d\omega_{B}^{n-1} = 0$~\cite{Gau77b}).
These metrics are intimately related to the Calabi-Yau-Gauduchon equation we study in this paper, see~\cite{FWW10a,FWW10b,TW17,GS25,Chu12,FLY12,CPY24}.

The Calabi-Yau-Gauduchon problem on a hermitian manifold $(M,\beta)$ equipped with a Gauduchon metric $\beta$ reduces to finding a hermitian metric $\alpha$ with the following properties (see~\cite[Theorem 1.3]{STW17}):

\begin{enumerate}

	\item $\alpha$ is Gauduchon, i.e., $\partial \overline{\partial}\alpha^{n-1} =0$.

	\item $[ \alpha^{n-1} ] = [ \beta^{n-1} ]$ in Aeppli cohomology $H^{n-1,n-1}_{A}(M)$. 

	\item $Ric^{BC}(\alpha) = \eta$, for a fixed closed real $(1,1)$ form $\eta \in c_{1}^{BC}(M)$.
	
\end{enumerate}

Recall that the $(n-1,n-1)$-Aeppli cohomology group is defined as:

\noindent
\[
H_{A}^{n-1,n-1}(M) = \frac{\{ \partial \overline{\partial}\text{-closed real } (n-1,n-1)\text{ currents}\}}{\{ \partial\gamma + \overline{\partial\gamma} ~~| ~~\gamma ~~ (n-2,n-1)\text{ current}\}}.
\]

By~\cite[Theorem A]{Pan22} for any compact smoothable hermitian variety $X$, that is a variety that can be approximated by a family of smooth manifolds as in~\ref{set: sing family setting}, 
there exists a singular Gauduchon metric, i.e., a hermitian metric $\beta$ that is smooth and Gauduchon on the regular part $X^{\mathrm{reg}}$ and bounded on $X$.
Such a metric is constructed as a limit of the metrics on the nearby smooth fibers of the family.
By~\cite[Theorem B]{Pan22} $\beta^{n-1}$ extendeds to a $(n-1,n-1)$ pluriclosed current $T_{G}$, i.e., $\partial \overline{\partial}T_{G} = 0$, defining an Aeppli class. (See~\cite{Pan25} for a generalization to compact hermitian varieties.)
Hence, both of the first points have their singular equivalents. Prescribing the Bott-Chern Ricci form on a singular variety is equivalent (see~\cite[Theorem 4.5]{GL23}) to solving

\noindent
\begin{equation}\label{eq: CY}\tag{CY}
	\omega^{n} = \mu
\end{equation}

\noindent
where $\mu$ is the adapted volume (see Definition~\ref{def: adapted volume})
Thus, one can pose the singular Calabi-Yau-Gauduchon problem as follows:

\begin{customthm}{Problem CYG} \label{prob: Problem CYG}
Let $X$ be a compact smoothable hermitian variety, 
$\beta$ a singular Gauduchon metric on $X$ with Gauduchon class $[\beta^{n-1} ] \in H^{n-1,n-1}_{A}(X)$ and
$\eta \in c_{1}^{BC}(X)$ a closed real $(1,1)$ form on $X$.
Does there exist a singular metric $\alpha$ on $X$ such that

\begin{enumerate}

	\item $\alpha$ is Gauduchon on $X$, i.e., $\partial \overline{\partial}\alpha^{n-1} =0$ on $X^{\mathrm{reg}}$.

	\item $[ \alpha^{n-1} ] = [ \beta^{n-1} ] \in H^{n-1,n-1}_{A}(X)$.

	\item $Ric^{BC}(\alpha) = \eta $ on $X^{\mathrm{reg}}$?
	
\end{enumerate}

\end{customthm}

\begin{customthm}{Remark}
The expected metric needs to be more singular than the \textit{bounded} Gauduchon metric from~\cite[Definition 1.2]{Pan22}. 
If the solution was quasi-isometric to a smooth metric then $X$ would be smooth by~\cite[Lemma 1.6]{GGZ23b}.
By~\cite[Lemma 6.4]{EGZ09} the volume of the desired Gauduchon metric needs to be $L^{p}$ relative to a smooth metric on $X$, for some $p>1$.	
\end{customthm}

Let us recall the notion of smoothing families:

\begin{customthm}{Setting (SF)}\label{set: sing family setting}

Let $\mathfrak{X}$ be an $(n+1)$-dimensional, irreducible, reduced, complex analytic space.
Suppose that 
\begin{itemize}
    \item $\pi: \mathfrak{X} \to \mathbb{D}$ is a proper, surjective, holomorphic map with connected fibers $X_{t} := \pi^{-1}(t)$;
    \item $\pi$ is smooth over the punctured disk $\mathbb{D}^{\ast}$;
    \item $X_{0}$ is an $n$-dimensional, normal compact complex analytic space. 
	\item $X_{0}$ has only isolated singularities.
\end{itemize}
Let $\omega, \beta$ be hermitian metrics on $\mathfrak{X}$ in the sense of Definition~\ref{def:metrics}.
For each $t \in \mathbb{D}$, we define the hermitian metric $\omega_{t}$ on fiber $X_{t}$ by restriction (i.e. $\omega_{t} := \omega|_{X_t}$).

\end{customthm}

Call $X$ smoothable if it arises as such a central fiber $X_{0} = X$ for some family $\mathfrak{X}$ as above.

In the smooth case \ref{prob: Problem CYG}
is solved in~\cite[Theorem 1.4]{STW17}
by studying a complex Monge-Ampère equation for $(n-1)$-plurisubharmonic functions with a gradient term.
Such equation was discovered by Popovici~\cite{Pop15} and independently by Tossati and Weinkove~\cite{TW19} (see~\cite{FWW10a,FWW10b} for earlier related discussions).
We will restrict ourself to the case without gradient term, studied in~\cite{TW17}, that is:

For $(\mathfrak{X},\beta,\omega)$ a smoothing as in~\ref{set: sing family setting},
fix $p>1$ and $0 \leqslant f_{t} \in L^{p}(X_{t},\omega_{t}^{n})$ a family of densities, $ 1 \leqslant \rho_{t} \in L^{\infty}(X_{t})$ a family of functions with $f_{t},\rho_{t} \in C^{\infty}(X_{t}), \forall t \neq 0$.
Assume the families $(f_{t})_{t \in \mathbb{D}}$ and $(\rho_{t})_{t \in \mathbb{D}}$ satisfy: there exists constants $c_{f},C_{f},G >0$ such that $\forall t \in \mathbb{D}$,

\noindent
\begin{equation}\label{eq: density bounds}\tag{DB}
	c_{f} \leqslant \int_{X_{t}}f_{t}^{\frac{1}{n}}\omega_{t}^{n} \quad \text{and} \quad ||f_{t}||_{L^{p}(X_{t},\omega_{t}^{n})} \leqslant C_{f},
\end{equation}

\noindent
\begin{equation}\label{eq: Gauduchon bound}\tag{GB}
	\sup_{X_{t}}\rho_{t} \leqslant G \quad \text{and} \quad \inf_{X_{t}} \rho_{t} = 1.
\end{equation}

Then we study the following equation with the conditions $\sup_{X_{t}}\varphi_{t} = 0$ and $\alpha_{t} >0$:

\noindent
\begin{equation}\label{eq: family CYG}\tag{CYG}
	\alpha_{t}^{n} := \left( \rho_{t} \beta_{t} + \frac{(\Delta_{\omega_{t}}\varphi_{t})\omega_{t} -dd^{c}_{t}\varphi_{t}}{n-1} \right)^{n} = c_{t}f_{t}\omega_{t}^{n},
\end{equation}

\noindent
where $(\varphi_{t},c_{t}) \in L^{\infty}(X_{t}) \cap \mathrm{PSH}_{n-1}(X_{t}, \rho_{t}\beta_{t}, \omega_{t}) \times \mathbb{R}_{>0}$ is the unkown.
The set $\mathrm{PSH}_{n-1}(X,\beta,\omega)$ of $(n-1)$-quasi plurisubharmonic functions on a space $X$ with two hermitian metrics $\beta,\omega$ can be defined (see Definition~\ref{def: q n-1 psh}) as the set of quasi-subharmonic functions
such that $\alpha_{\varphi} := \beta + \frac{(\Delta_{\omega}\varphi)\omega -dd^{c}\varphi}{n-1} \geqslant 0$ in the sense of currents.

By~\cite[Lemma 4.4]{DNGG23} that the adapted volume of smoothings satisfy these integral bounds~\eqref{eq: density bounds} 
when the central fiber $X_{0}$ has canonical singularities.
By~\cite[Theorem A]{Pan22} the Gauduchon factors of $\beta_{t}$ satisfy the uniform bound~\eqref{eq: Gauduchon bound}.
Hence, then a $L^{\infty}$ bound for the problem above gives a bound for a simplified singular Calabi-Yau-Gauduchon.

\smallskip
We can state now our first result:

\begin{customthm}{Theorem A}\label{thm: C0 family}

Under~\ref{set: sing family setting}
and the uniform hypotheses above~\eqref{eq: density bounds}~\eqref{eq: Gauduchon bound},
For each $t \in \mathbb{D}$, let the pair $(\varphi_{t},c_{t}) \in \mathrm{PSH}_{n-1}(X_{t},\beta_{t},\omega_{t})  \cap L^{\infty}(X_{t}) \cap \mathbb{R}_{>0}$ be the solution to the complex $(n-1)$Monge-Ampère equation~\eqref{eq: family CYG}.

Then there exists a constant $C>0$ uniform, independent of $t$, such that~$\forall t \in \mathbb{D}_{1/2}$,

\noindent
\[
c_{t} + c_{t}^{-1} + ||\varphi_{t}||_{L^{\infty}} \leqslant C.
\]

\end{customthm}

For the precise dependence of $C$ above see Theorem~\ref{teo: C0 Vincent}.

\smallskip
The higher order estimates, obtained in~\cite{TW17} and~\cite{Pan22} are non quantitative. 
The first due to a blow-up analysis for the $C^{1}$-estimate.
The second due to a localization argument. 
Consequently, obtaining higher order estimates in the presence of singularities for Equation~\eqref{eq: family CYG} is more daunthing.
Taking this into account, we dedicate the rest of the paper to the case of families with no singular fiber, that is
a deformation of the complex structure of a complex manifold $M$, here assumed Kähler. 

To highlight the fact that for this second part there is no singular fiber 
we will represent our complex manifolds with $M$ and the holomorphic family with $\mathcal{M}$.

In such a setting we have our second result:

\begin{customthm}{Theorem B}\label{teo: estimates HF}

A holomorphic family $(\mathcal{M}, \beta, \omega)$ under the assumptions of~\ref{set: sing family setting}, with the central fiber $M_{0}$ a smooth Kähler manifold and $\omega$ a Kähler metric, 
$f \in C^{\infty}(\mathcal{M})$ such that for each $t \in \mathbb{D}$ let $f_{t} := f|_{M_{t}} \in C^{\infty}(M_{t})$, where $M_{t}:= \pi^{-1}(t)$. 
Then, for each $t \in \mathbb{D}_{\frac{1}{2}}$ there exists a unique solution $(u_{t},c_{t}) \in \mathrm{PSH}_{n-1}(X,\beta, \omega) \cap C^{\infty}(M_{t}) \times \mathbb{R}_{>0}$ for the equation:

\noindent
\[
\left(\beta_{t} + \frac{1}{n-1}\Big( (\Delta_{\omega_{t}}u_{t})\omega_{t} - dd^{c}u_{t} \Big) \right)^{n} = c_{t}f_{t}\beta_{t}^{n}, \quad \sup_{M_{t}}~u_{t} = 0
\]
	
\noindent
and for any $k \in \mathbb{N}$

\noindent
\[
||u_{t}||_{C^{k}(M_{t})} + c_{t} + c_{t}^{-1} \leqslant C
\]

\noindent
for a uniform $C>0$ independent of $t$. Hence, $C$ independs on the complex structure of $M$.

\end{customthm}

This past result produces in a general Kähler manifold $M$ non-Kähler Gauduchon (or semi-Kähler) metrics,
with prescribed Ricci curvature, that vary smoothly with the complex structure if one assumes $\beta$ is non-Kähler Gauduchon (or semi-Kähler) 
as they have $C^{\infty}$ bounds independent of it. See Section~\ref{CYG: preliminaries} or~\cite{TW17}.

\textbf{Organization of the article:}

\begin{itemize}
	\item Section~\ref{CYG: preliminaries}
	familiarizes the reader with $(n-1)$-plurisubharmonic functions in $\mathbb{C}^{n}$, manifolds and introduces the singular setting.
	
	\item Section~\ref{CYG: C0}
	stablishes two different uniform bounds for the normalizing constant and proves~\ref{thm: C0 family}.
	
	\item Sections~\ref{CYG: C2} and~\ref{CYG: C1 e HO HF} extend the $C^{2},C^{1}$ and higher-order estimates of~\cite{TW17} to holomorphic families
	and complete the proof of~\ref{teo: estimates HF}
\end{itemize}

\begin{customthm}{Acknowledgements}
The author would like to first thank his PhD advisor, 
Vincent Guedj, for his support, guidance and comments. 
The author also thanks Chung-Ming Pan and Tat~Dat~Tô for useful discussions and references.
Furthermore, the author thanks Hoang-Chinh Lu for discussions regarding the Hessian equation.  
This work received support from the EUR-MINT
(State support managed by the National Research 
Agency for Future Investments
program bearing the reference ANR-18-EURE-0023).

\end{customthm}

\section{Quasi \texorpdfstring{$(n-1)$}{n-1}-plurisubharmonic functions}\label{CYG: preliminaries}

We set $d^{c} = \frac{i}{2}(\overline{\partial} - \partial) $; hence $dd^{c} = i \partial \overline{\partial}$, $\omega_{\mathbb{C}^{n}}$ the canonical metric of $\mathbb{C}^{n}$.
Fix $n \geqslant 3$ the complex dimension of our spaces, since $(n-1)$-psh functions are just psh functions for $n=2$.

We define the complex Laplacian with respect to a hermitian metric $\omega$ as:

\noindent
\[
\Delta_{\omega}g := tr_{\omega}(dd^{c}g) = n \frac{dd^{c}g \wedge \omega^{n-1}}{\omega^{n}}.
\]

\noindent
\subsection{\texorpdfstring{$(n-1)$}{n-1}-plurisubharmonic functions}

Fix $\Omega \subset \mathbb{C}^{n}$ domain.

\subsubsection{Basic properties and examples}

\begin{definition} \label{def: n-1 psh}
We say that $u \in C^{2}(\Omega)$ is a $(n-1)$-plurisubharmonic function on $\Omega$
if the folowing $(1,1)$-form is semi-positive

\noindent
\[
\widehat{dd^{c}u} := (\Delta u)\omega_{\mathbb{C}^{n}} - dd^{c}u \geqslant 0.
\]

\end{definition}

Note that $\widehat{dd^{c} (\cdot)}$ is a linear operator.
One has the following equivalent definitions which can be found throughout the literature, see~\cite{HL11, HL12, Ver10, Din22}.

\begin{proposition} \label{prop: equiv def}
Let $u \in C^{2}(\Omega)$. Then $u$ is $(n-1)$-plurisubharmonic on $\Omega$ 
if and only if one of the following equivalent conditions is satisfied:

\begin{enumerate}
    \item The eigenvalues $\{ \lambda_{i} \}_{i =1}^{n}$ of the complex hessian of $u$, satisfy $\widehat{\lambda_{i}} := \sum_{k \neq i} \lambda_{k} \geqslant 0$.
    
    \item For any complex hyperplane of dimension $n-1$, $H \subset \mathbb{C}^{n}$, $u|_{H \cap \Omega}$ is subharmonic.
    
    \item The $(n-1,n-1)$-form $dd^{c}u \wedge \omega_{\mathbb{C}^{n}}^{n-2}$ is (weakly) positive.
    
    \item The $(n-1,n-1)$-form $dd^{c}u \wedge \omega_{\mathbb{C}^{n}}^{n-2}$ is strongly positive.
\end{enumerate}

\end{proposition}

We refer the reader to~\cite[Chapter III, Corollary 1.9]{DemBook} for $(3) \Longleftrightarrow (4)$.
We highlight that $(3) \Longleftrightarrow u$ is $(n-1)$-psh comes from the Hodge star computation:

\noindent
\[
\frac{1}{(n-1)!} *(dd^{c}u \wedge \omega_{\mathbb{C}^{n}}^{n-2}) = \frac{1}{n-1}((\Delta u)\omega_{\mathbb{C}^{n}} - dd^{c} u ) = \frac{1}{n-1} \widehat{dd^{c} u}.
\]

The conventions used in this note for the Hodge star operator follow~\cite[Section 2]{TW17}, that is: 
Fix $p,q \in \{ 1, \ldots, n \}$, the Hodge star operator $*$ (induced by $\omega_{\mathbb{C}^{n}}$ in this subsection)
maps a $(p,q)$-form $\psi$ into a $(n-p, n-q)$-form $*\psi$ such that for any $(p,q)$-form $\varphi$:

\noindent
\begin{equation} \label{eq: hodge star}
\varphi \wedge *\overline{\psi} = \langle \varphi, \psi \rangle_{\omega_{\mathbb{C}^{n}}} \frac{\omega_{\mathbb{C}^{n}}^{n}}{n!}, \quad \overline{*\psi} = *\overline{\psi}, \quad **\psi = (-1)^{p+q}\psi,
\end{equation}

From Proposition~\ref{prop: equiv def} one notices that $u$ is subharmonic, hence we define $m$-subharmonic functions to compare these notions of positivity.

\begin{definition} \label{def: m subharmonic}
Let $1 \leqslant m \leqslant n$. 
We say $u \in C^{2}(\Omega)$ is $m$-subharmonic on $\Omega$ if

\noindent
\[
(dd^{c}u)^{k} \wedge \omega_{\mathbb{C}^{n}}^{n-k} \geqslant 0, \quad \forall 1 \leqslant k \leqslant m.
\]

\end{definition}

We denote $\mathrm{PSH}_{n-1}(\Omega)$ the set of $(n-1)$-plurisubharmonic ($(n-1)$-psh for short)
functions in $\Omega$, $\mathrm{SH}_{m}(\Omega)$ the set of $m$-subharmonic functions and $\mathrm{PSH}(\Omega) = SH_{n}(\Omega)$ the plurisubharmonic functions in $\Omega$.
All these sets exclude the function $u \equiv -\infty$. We introduce non-regular
$(n-1)$-psh functions:

\begin{definition} \label{def: w n-1 psh}

Let $u \in \mathrm{SH}_{1}(\Omega)$ be a subharmonic function on $\Omega$. 
Then $u \in \mathrm{PSH}_{n-1}(\Omega)$ if:

\noindent
\[
\widehat{dd^{c}u} := (\Delta u)\omega_{\mathbb{C}^{n}} - dd^{c}u \geqslant 0
\]

as a real current of bidegree $(1,1)$.
    
\end{definition}

\begin{proposition} \label{prop: basic properties}
Let $u,v \in \mathrm{PSH}_{n-1}(\Omega)$ and $\lambda >0$. Then,

\begin{itemize}
    \item[(i)] $u + v, \lambda u, \max\{u,v\} \in \mathrm{PSH}_{n-1}(\Omega)$.
    
    \item[(ii)] If $(w_{i})_{i=1}^{+\infty} \subset \mathrm{PSH}_{n-1}(\Omega)$ and $w_{i} \searrow w \not\equiv - \infty$ pointwise on $\Omega$, then $w \in \mathrm{PSH}_{n-1}(\Omega)$.

    \item[(iii)] If $u$ is continuous and $\chi_{\varepsilon}$ is a family of mollifiers,
    then their convolution $u_{\varepsilon} = u \star \chi_{\varepsilon}$ is smooth $(n-1)$-psh function and converges to $u$ locally uniformly.
    
    \item[(iv)] $\mathrm{PSH}(\Omega) \subset \mathrm{SH}_{2}(\Omega) \subset \mathrm{PSH}_{n-1}(\Omega) \subset \mathrm{SH}_{1}(\Omega)$.

\end{itemize}

\end{proposition}

The proofs for (i)-(iii) above are analogous to the classic ones for subharmonic functions regardless of the definition chosen.
The inclusion $\mathrm{SH}_{2}(\Omega) \subset \mathrm{PSH}_{n-1}(\Omega)$ holds since the eigenvalues $\{\lambda_{i}\}_{i=1}^{n}$ of a smooth function in $\mathrm{SH}_{2}(\Omega)$ satisfy $ (\sum_{i = 1}^{n}\lambda_{i})^{2} - \sum_{i=1}^{n} \lambda_{i}^{2} \geqslant 0 \Longrightarrow \sum_{i=1}^{n} \lambda_{i} \geqslant \lambda_{j}, \forall j \in \{ 1, \ldots, n \}$ since $\sum_{i=1}^{n}\lambda_{i} \geqslant 0$.

\begin{example} \label{ex: basic (n-1)psh}
\begin{enumerate}
    \item $g_{\varepsilon}(z) = - \frac{1}{(|z|^{2} + \varepsilon^{2})^{(n-2)}} \in C^{\infty}(\mathbb{C}^{n}) \cap \mathrm{PSH}_{n-1}(\mathbb{C}^{n})$ for $\varepsilon >0$. Notice that $g_{\varepsilon} \searrow G_{n-1}(z) = - \frac{1}{|z|^{2(n-2)}} \in \mathrm{PSH}_{n-1}(\mathbb{C}^{n})$ and $G_{n-1}(0) = - \infty$. Recall that $n \geqslant 3$.\footnote{One can also prove $G_{n-1} \in \mathrm{PSH}_{n-1}(\mathbb{C}^{n})$ by~\cite[Theorem 9.2]{HL14} as the origin is a removable singularity.}

    \item $u:\mathbb{C}^{3} \to \mathbb{R}$, $u(z_{1}, z_{2}, z_{3}) = - |z_{1}|^{2} + |z_{2}|^{2} + |z_{3}|^{2}$ is such that $u \in \mathrm{PSH}_{n-1}(\mathbb{C}^{3}) \backslash \mathrm{SH}_{2}(\mathbb{C}^{3})$.
    
    \item $v:\mathbb{C}^{3} \to \mathbb{R}$, $v(z_{1}, z_{2}, z_{3}) = - \frac{3}{2}|z_{1}|^{2} + |z_{2}|^{2} + |z_{3}|^{2}$ is such that $v \in \mathrm{SH}_{1}(\mathbb{C}^{3}) \backslash \mathrm{PSH}_{n-1}(\mathbb{C}^{3})$.
\end{enumerate}

\end{example}

Now we define the complex $(n-1)$Monge-Ampère operator $\widehat{MA}$ as

\begin{definition}\label{def: n-1 MA}
Given $u \in C^{2}(\Omega)\cap \mathrm{PSH}_{n-1}(\Omega)$, we set

\noindent
\[
\widehat{MA}(u) := \det (\widehat{dd^{c}u})
\]

\noindent
in terms of eigenvalues we have

\noindent
\[
\widehat{MA}(u) = \prod_{i=1}^{n} \widehat{\lambda_{i}} = \prod_{i=1}^{n} \left( \sum_{j \neq i} \lambda_{j} \right)
\]

\noindent
where $\{ \lambda_{i} \}_{i=1}^{n}$ are the eigenvalues of the complex hessian of $u$.
\end{definition}

\begin{remark}
Let's stress that the operator $\widehat{MA}$ does not satisfy some properties of the complex Monge-Ampère operator such as the comparison principle by~\cite[Theorem 4.1]{Din22}.
This makes the development of a pluripotential theory for this operator more delicate.
\end{remark}

The notion of maximality for $(n-1)$-plurisubharmonic functions, 
crucial for the blow-up analysis that will be constructed in Section~\ref{CYG: C1}.

\begin{definition}\label{def: maximal n-1 psh}
We say $u \in \mathrm{PSH}_{n-1}(\Omega)$ is maximal if for any relatively compact open set $\Omega_{1} \Subset \Omega$
and $\forall v \in \mathrm{PSH}_{n-1}(\Omega_{2})$ such that $\Omega_{1} \Subset \Omega_{2} \Subset \Omega$
and $v \leqslant u$ on $\partial \Omega_{1}$
then $v \leqslant u$ on $\Omega_{1}$.
\end{definition}

\smallskip

\subsubsection{Radial Examples}

\begin{lemma} \label{lema: Radial cond}
Let $G_{n-1}:\mathbb{C}^{n} \to \mathbb{R}$ as in Example~\ref{ex: basic (n-1)psh}
and $\chi:\mathbb{R}_{\leqslant 0} \to \mathbb{R}$ smooth. Then $u(z) = \chi(G_{n-1}(z)) \in \mathrm{PSH}_{n-1}(\mathbb{C}^{n})$ if and only if $\chi$ is non-decreasing and convex.
\end{lemma}

\begin{proof}
A direct computation shows that the eigenvalues of the complex hessian of $u$ are:

\noindent
\[
\lambda_{1} = (n-2)|z|^{-2(n-1)} \big[(n-2)|z|^{-2(n-2)} \chi^{\prime \prime} - (n-2) \chi^{\prime} \big], \quad \lambda_{j} = (n-2) |z|^{-2(n-1)} \chi^{\prime}, \quad \forall j>1,
\]

\noindent
hence, the eigenvalues of of $\widehat{dd^{c}u}$ are:

\noindent
\[
\widehat{\lambda}_{1} = (n-1)(n-2) |z|^{-2(n-1)} \chi^{\prime}, \quad \widehat{\lambda}_{j} = (n-2)^{2} |z|^{-2(2n-3)} \chi^{\prime \prime}; \quad \forall j>1.
\]

This concludes the proof.
\end{proof}

By the Lemma~\ref{lema: Radial cond} above $\widehat{MA}$ for radial functions takes the form:

\noindent
\[
\widehat{MA}(u) = (n-1)(n-2)^{2n-1} |z|^{-4(n-1)^{2}} (\chi^{\prime} \circ G_{n-1}) (\chi^{\prime \prime} \circ G_{n-1})^{n-1}. 
\]

Example~\ref{ex: logloglog n-1 MA} bellow shows a gap between the integrability of $\widehat{MA}(u)$ and the classical Monge-Ampère case for the $L^{1}(\log L)^{p}$-norm.

\begin{example} \label{ex: logloglog n-1 MA}
Let $\chi(t) = - \log ( \log ( \log (-t)))$, computing near zero we obtain:

\noindent
\[
\chi^{\prime}(t) = \frac{1}{(-t) \log(-t) \log ( \log(-t))}, \quad \chi^{\prime \prime}(t) \sim - \frac{1}{(-t)^{2} \log(-t) \log( \log(-t))}.
\]

Hence, the density $f := \widehat{MA}(u)$, for $u(z) = \chi(G_{n-1}(z))$ satisfies near:

\noindent
\[
f \sim \frac{1}{|z|^{2n} (-\log (|z|))^{n} (\log ( -\log (|z|)))^{n}}, \quad \text{and} \quad \log(f) \sim -\log(|z|),
\]

\noindent
then near zero one has

\noindent
\[
\int^{+\infty} f (\log (f))^{p} \sim \int_{0} \frac{1}{|z|^{2n} (-\log|z|)^{n-p} (\log ( - \log|z|))^{n}}dV_{\mathbb{C}^{n}} 
\]

\noindent
which is finite if and only if $p \leqslant n-1$. However, for the Monge-Ampère operator it is known that the critical exponent is $p=n$, see~\cite[Example 2.4]{GL25}.
\end{example}

\begin{remark}
A computation shows $\left(\widehat{MA}(G_{n-1})\right)^{\frac{1}{n-1}} = \delta_{0}$, the Dirac mass at the origin.
\end{remark}

\subsection{Quasi \texorpdfstring{$(n-1)$}{n-1}-plurisubharmonic functions on manifolds}

Throughout this paper $(M,\beta, \omega)$ will denote a Hermitian manifold with two hermitian metrics.
Fix a geometric constant $C_{\beta,\omega}>1$ such that 

\noindent
\begin{equation}\label{cst: geo constant}
\tag{GC}
C_{\beta,\omega}^{-1}\beta \leqslant \omega \leqslant C_{\beta,\omega}\beta.
\end{equation}

\subsubsection{Smooth case}

\begin{definition}\label{def: q n-1 psh}
Let $u \in C^{2}(M)$, we say $u$ is a quasi $(n-1)$-plurisubharmonic function relative to $(\beta, \omega)$, if the following form is semi-positive

\noindent
\[
\alpha_{u} := \beta + \frac{(\Delta_{\omega}u)\omega -dd^{c}u}{n-1} \geqslant 0.
\]

\end{definition}

We let $C^{2}(M) \cap \mathrm{PSH}_{n-1}(M, \beta, \omega)$ denote the set of ($C^{2}$)-smooth quasi $(n-1)$-plurisubharmonic functions relative to $(\beta,\omega)$,
$\mathrm{PSH}_{n-1}(M,\omega)$ when $\beta = \omega$,
$\mathrm{SH}_{1}(M,\beta,\omega)$ the set of quasi-subharmonic functions relative to $(\beta,\omega)$
and $\mathrm{PSH}(M,\omega)$ the set of quasi-plurisubharmonic functions relative to $\omega$. 
All of these sets exclude the function $u \equiv -\infty$.

Taking the trace of $\alpha_{u}$ with respect to $\omega$ we explain in Proposition~\ref{prop: MA Lap1}
that $u$ as above is quasi-subharmonic with respect to $(\beta, \omega)$.
For basic properties of these functions we we refer to~\cite{HL11,HL12,ADO22,AO22,TW17,Pop15}.

\smallskip

The operator of interest, namely the complex $(n-1)$Monge-Ampère operator $MA_{n-1}$, is:

\begin{definition}\label{def: n-1MA}

Given $u \in C^{2}(M)\cap \mathrm{PSH}_{n-1}(M,\beta,\omega)$, we {define} $A_{u}: T^{1,0}M \to T^{1,0}M$ the endomorphism defined by $\alpha_{u}$
relative to $\omega$, i.e., $A_{u} :=\omega^{-1} \cdot \alpha_{u}$. And set

\noindent
\[
MA_{n-1}(\beta, \omega, u) := \det (A_{u}).
\]

\end{definition}
Notice that for any constant $C>0$, $MA_{n-1}(\beta, C\omega, u) = C^{-n} MA_{n-1}(\beta, \omega, u)$.

\begin{remark}
Let $\Phi_{u} = \omega_{0}^{n-1} + dd^{c} u \wedge \omega^{n-2}$, for $\omega_{0}$ is a Gauduchon non-Kähler metric on $M$.
Then $\Phi_{u}^{\frac{1}{n-1}}$ is a Gauduchon non-Kähler metric\footnote{See~\cite[Lemma 4.8]{Mic82} for the existence and uniqueness of the $(n-1)$th root.}
and $[\omega_{0}^{n-1}] = [\Phi_{u}]$ in Aeppli cohomology for any $u \in C^{2}(M) \cap \mathrm{PSH}_{n-1}(M,\omega_{0}, \omega)$, since $\omega$ is Kähler
and by the equivalence of definitions for $(n-1)$-psh on Proposition~\ref{prop: equiv def}.
By~\cite[Equation (2.2)]{TW17} $\frac{\det(\Phi_{u})}{\det(\omega^{n-1})} = \frac{\alpha_{u}^{n}}{\omega^{n}}$, hence solving the 
$(n-1)$-Monge-Ampère equation for $(1,1)$ forms is equivalent to solving it for $(n-1,n-1)$ forms which is equivalent to solving the Gauduchon conjecture in this setting.
\end{remark}

\subsubsection{Non-smooth case}

\begin{definition}\label{def: w q n-1 psh}

Let $u \in \mathrm{SH}_{1}(M, \beta, \omega)$ be a quasi-subharmonic function on $M$. 
We say $u$ is a {quasi $(n-1)$-plurisubharmonic function} relative to $(\beta, \omega)$ if:

\noindent
\[
\alpha_{u} := \beta + \frac{(\Delta_{\omega}u)\omega -dd^{c}u}{n-1} \geqslant 0
\]

as a current of bidegree $(1,1)$.

\end{definition}

This definition is equivalent to those in~\cite{HL12}. We let $\mathrm{PSH}_{n-1}(M, \beta, \omega)$ be the set of all, not identically $- \infty$, quasi $(n-1)$-plurisubharmonic functions relative to $(\beta, \omega)$.

\begin{remark}
We also refer the reader to~\cite{CX25} for a viscosity approach
to the weak theory of $(n-1)$-plurisubharmonicity and $(n-1)$Monge-Ampère.

\end{remark}

\subsection{Analysis on complex analytic spaces}

We let $X$ be a compact normal complex
analytic space of pure dimension $n \geqslant 1$.
 We denote by $X^{\mathrm{reg}}$ the 
 complex manifold of regular points of $X$ and 
$
X^{\mathrm{sing}}:=X \setminus X^{\mathrm{reg}}
$
the set of singular points, which is an analytic subset 
of $X$ with complex codimension $\geqslant 2$, since $X$ is normal.

By definition, for each point $x_{0} \in X$, there exists 
a neighborhood $U$ of $x_{0}$ and a local embedding 
$j: U \hookrightarrow \mathbb{C}^N$ onto an analytic 
subset of $\mathbb{C}^N$ for some $N \geqslant 1$.
These local embeddings allow us to define the 
spaces of smooth forms of given degree on $X$ as smooth 
forms on $X^{\mathrm{reg}}$ that are locally on $X$  
restrictions of an ambient form on $\mathbb{C}^N$. 
Other differential notions and operators, such as 
holomorphic,plurisubharmonic and $(n-1)$-plurisubharmonic functions,
can also be defined in this way and
currents are defined by duality (see~\cite{Dem85}).

\begin{definition}\label{def:metrics}
A {hermitian metric} $\omega$ on $X$ is a 
hermitian metric $\omega$ on $X^{\mathrm{reg}}$ such that 
given any local embedding $X \underset{\mathrm{loc}.}{\hookrightarrow} \mathbb{C}^{N}$, $\omega$ 
extends smoothly to a hermitian metric on $\mathbb{C}^{N}$.
\end{definition}

Following the definition introduced in~\cite[Definition 1.2]{Pan22} we set:

\begin{definition}\label{def: bounded Gauduchon}

    A hermitian metric $\omega_{G}$ on a variety $X$ is:

    \begin{itemize}
        \item {Gauduchon} if it satisfies $dd^{c}(\omega_{G}^{n-1}) = 0$ on $X^{\mathrm{reg}}$;
        
        \item {bounded Gauduchon} if there exists a smooth hermitian metric $\omega$ on $X$
        and a positive function $\rho \in L^{\infty}(X) \cap C^{\infty}(X^{\mathrm{reg}})$
        such that $\omega_{G} = \rho \omega$ and $dd^{c}(\rho\omega)^{n-1} = 0$ on $X^{\mathrm{reg}}$.\footnote{Our convention for the Gauduchon factor differs from the one of Pan by the absence of the $\frac{1}{n-1}$ power.}
    \end{itemize}
    
\end{definition}

\subsubsection{Adapted volume and canonical singularities}

\begin{definition}\label{def: adapted volume}
    
    Suppose that $X$ is $\mathbb{Q}$-Gorenstein, that is $X$ is normal and $K_{X}$ is $r$-Cartier for some $r \in \mathbb{N}$, i.e., $rK_X$ is a rank 1 locally free sheaf.
If $\sigma$ is a trivialization of $rK_X$ over $X^{\mathrm{reg}}$ and $h^{r}$ a smooth metric on $rK_X$ we consider the adapted volume form:

\noindent
\[
\mu_h := \left( \frac{i^{rn^{2}}\sigma \wedge \overline{\sigma}}{|\sigma|_{h^{r}}^{2}} \right)^{\frac{1}{r}}
\]

\end{definition}

It is independent of the choice of trivialization $\sigma$.
The term ”variety” will always refer to a $\mathbb{Q}$-Gorenstein complex analytic space throughout this paper.

\begin{definition}\label{def: canonical singularities}
A hermitian variety $X$ has canonical singularities if for any
resolution of singularities $p: \widetilde{X} \to X$, for any local generator
$\tau$ of $rK_{X}$, the meromorphic pluricanonical form $p^{*}\tau$ is holomorphic.
    
\end{definition}

Now we define the (Bott-Chern) Ricci curvature in this context

\begin{definition}\label{def: Ricci curvature}
    Let $X$ be a compact hermitian variety, $\mu$ the reference adapted volume relative to $h^{r} \equiv 0$ on local trivializations of $rK_{X}$, 
    we define the (Bott-Chern) {Ricci curvature} current of a metric $\omega$, for $\mu = f \omega^{n}$, as

    \noindent
    \[
    Ric^{BC}(\omega) := - dd^{c}\log \left( \frac{\omega^{n}}{\mu} \right) = -dd^{c}\log(f)  
    \]
\end{definition}

Let us stress that, by~\cite[Lemma 1.6]{GGZ23b}, the Ricci curvature is not a smooth form unless $X$ is smooth, see~\cite[Sections 5, 6 and 7]{EGZ09} for more details.

\subsubsection{Smoothable varieties}

A smoothable variety is a hermitian variety which is the central fiber of some family satisfying~\ref{set: sing family setting}.
Set $\mathbb{D}_{r} := \{ z \in \mathbb{C} : |z| < r \}$, $\mathbb{D}_{r}^{*} := \mathbb{D}_{r} \backslash 0$ and $\mathbb{D} := \mathbb{D}_{1}$.
The volumes of $(X_{t}, \omega_{t})$ are uniform,
there is a uniform constant $C_{V} \geqslant 1$ such that:

\noindent
\begin{equation}\label{eq: volume bound}
C_{V}^{-1} \leqslant \mathrm{Vol}_{\omega_{t}}(X_{t}) \leqslant C_{V}, \quad \forall t \in \overline{\mathbb{D}}_{1/2}.
\tag{VB}
\end{equation}

For non-Kähler examples of smoothable varieties one can look at~\cite[Theorem 1]{LT94}
and for examples of non-smoothable singularities one can look at~\cite[Section 0.1]{Har74}.

Recall that by~\cite[Lemma 4.4]{DNGG23} the adapted volume of smoothings satisfy these integral bounds~(\ref{eq: density bounds})
when the central fiber $X_{0}$ has canonical singularities 
and the Gauduchon factors of $\beta_{t}$ satisfy, by~\cite[Theorem A]{Pan22}, the uniform bound~(\ref{eq: Gauduchon bound}).

\subsection{Useful Results}

The following result is a simplified version of~\cite[Proposition 3.3]{Pan23} a global uniform Skoda integrability result. 
It will be relevant for the bound on Proposition~\ref{prop: bound c1}.

\begin{proposition}\label{prop:global_Skoda}
	If $\pi: (\mathfrak{X},\omega) \to \mathbb{D}$ is a family satisfying~\ref{set: sing family setting}.
	Then there exists constants $\alpha$, $A_\alpha$ such that for all $t \in \overline{\mathbb{D}}_{1/2}$ and for all $u_t \in \mathrm{PSH}(X_t, \omega_t)$ with $\sup_{X_t} u_t = 0$,

\noindent
\begin{equation}\label{eq: global Skoda}
    \int_{X_t} e^{-\alpha u_t} \omega_t^n \leqslant A_\alpha.
\tag{SB}
\end{equation}

\end{proposition}

We state next~\cite[Theorem 5.2]{TW17} the Liouville theorem for $(n-1)$-plurisubharmonic functions.
It will be crucial for the blow-up argument in Theorem~\ref{teo: C1 hol families}.

\begin{theorem}\label{teo: liouville thm n-1 psh}

If $u: \mathbb{C}^{n} \to \mathbb{R}$ is an $(n-1)$-plurisubharmonic function in $\mathbb{C}^{n}$
which is Lipschitz continous, maximal and satisfies

\noindent
\[
\sup_{\mathbb{C}^{n}} (|u| + | \nabla u |) < +\infty
\]

\noindent
then $u$ is a constant.
    
\end{theorem}

\section{\texorpdfstring{$L^{\infty}$}{Linfty} a priori estimate on smoothable Hermitian varieties}\label{CYG: C0}

\subsection{Comparison results and useful facts}

Now we relate the different notions of positivity and prove a few useful facts about $(n-1)$-plurisubharmonic functions.

\begin{proposition}\label{prop: omega omegah1}
For $(M,\beta, \omega)$ hermitian manifold with two hermitian metrics and a constant $C_{\beta,\omega}> 0$ such that $C_{\beta, \omega}^{-1}\omega \leqslant \beta$. 
Then $\mathrm{PSH}_{n-1}(M,C_{\beta,\omega}^{-1}\omega) \subset \mathrm{PSH}_{n-1}(M,\beta,\omega)$.
\end{proposition}

\begin{proof}
We have the following computation:

\noindent
\begin{align*}
    \alpha_{u} = \beta + \frac{1}{n-1}((\Delta_{\omega}u)\omega - dd^{c}u) & =  \beta + C_{\beta, \omega}^{-1}\omega - C_{\beta, \omega}^{-1}\omega +\frac{1}{n-1}((\Delta_{\omega}u)\omega - dd^{c}u),\\
    &= (\beta - C_{\beta, \omega}^{-1}\omega) + (C_{\beta, \omega}^{-1}\omega + \frac{1}{n-1}((\Delta_{C_{\beta, \omega}^{-1}\omega}u)C_{\beta, \omega}^{-1}\omega - dd^{c}u) ),\\
    & \geqslant C_{\beta, \omega}^{-1}\omega + \frac{1}{n-1}((\Delta_{C_{\beta, \omega}^{-1}\omega}u)C_{\beta, \omega}^{-1}\omega - dd^{c}u).
\end{align*}

As we wanted to prove.
\end{proof}

When doing calculations in coordinates we will frequentely use,
for a fixed point $x \in M$ a chart $U_{x}$ around $x$ such that
$\omega_{i\bar{\jmath}}(x) = \delta_{i\bar{\jmath}}$, $\beta(x)= \gamma_{i}\delta_{i\bar{\jmath}}$ and $dd^{c}u(x) = \lambda_{i}\delta_{i\bar{\jmath}} + \text{O.D.T.}$
for some smooth function $u$, where O.D.T. stands for "off-diagonal terms''(similarly to~\cite{TW17}) containing $dz^{k}\wedge d\overline{z}^{l}$ for $k \neq l$. 
However those terms won't affect our calculations.

Moreover, the condition of $u \in \mathrm{PSH}_{n-1}(M,\beta,\omega)$ can be interpreted as:

\noindent
\[
\widehat{\lambda}_{i} := \sum_{\underset{j\neq i}{j=1}}^{n} \lambda_{j} \geqslant -(n-1)\gamma_{i}, \quad \forall i \in \{ 1,\dots, n \}.
\]

We first prove a domination principle for $\mathrm{PSH}_{n-1}(M,\beta,\omega)$ as in~\cite[Corollary 2.3]{GL25H}.
Our proof is simpler as we only need to deal with smooth functions: 

\begin{lemma}\label{lema: DP1}
Let $u,v \in \mathrm{PSH}_{n-1}(M,\beta, \omega) \cap C^{\infty}(M)$ such that
$\alpha_{u}^{n} \leqslant c \alpha_{v}^{n}$ on $\{ u<v \}$
for some $c \in [0,1)$. Then $u\geqslant v$.
\end{lemma}

\begin{proof}

    Assume by contradiction that $\{u<v\}$ is non-empty and fix $x_0 \in M$ such that
    
    \noindent
    $$
    (u-v)\left(x_0\right)=\min _M(u-v)<0 .
    $$
    
    If $a \in(c, 1)$ and $x$ is a point on $M$ such that $(u-a v)(x)=\min _X(u-a v)$, then
    
    \noindent
    $$
    \begin{aligned}
    u(x)-v(x) & =(u-a v)(x)+(a-1) v(x) \\
    & \leq(u-a v)\left(x_0\right)+(a-1) v(x) \\
    & \leq(u-v)\left(x_0\right)+(a-1)\left(v(x)-v\left(x_0\right)\right) \\
    & \leq(u-v)\left(x_0\right)+(1-a) \operatorname{Osc}_M(v)
    \end{aligned}
    $$
    
    Since $(u-v)\left(x_0\right)<0$, we can fix $a \in(c, 1)$, sufficiently close to 1 , such that
    $$
    (u-v)\left(x_0\right)+(1-a) \operatorname{Osc}_M(v)<0
    $$
    
    Then, from the above, we can find $x_a \in\{u<v\}$ such that $(u-a v)\left(x_a\right)=$ $\min _M(u-a v)$. The classical maximum principle ensures that at $x_a$ we have $d d^c(u-a v) \geq 0$, implying
    $$
     \alpha_{u}  \geq a \alpha_{v} + (1-a)\beta
    $$
 
\noindent
because $(\Delta_{\omega}(u-av))\omega - dd^{c}(u-av) \geqslant 0$.
Since these forms are positive, we obtain
    $$
    \left( \alpha_{u} \right)^n \geqslant a^n c^{-n}\left( \alpha_{u} \right)^n + (1-a)^n \beta^{n}
    $$
    a contradiction, because $a c^{-1} \geq 1$, and $(1-a) \beta^n>0$ at $x_a$.
\end{proof}

The following consequences of the arithmetic-geometric mean inequality will be useful:

\begin{proposition}\label{prop: MA Lap1}
Given a smooth $u\in \mathrm{PSH}_{n-1}(M,\beta,\omega)$ such that it solves the complex $(n-1)$Monge-Ampère equation
$\alpha_{u}^{n} = bg\omega^{n}$ for a function $0<g\in C^{\infty}(M)$ and some constant $b>0$.
Then $u \in \mathrm{SH}_{1}(M,\beta,\omega)$
and the following inequality holds:

\noindent
\[
(n b^{\frac{1}{n}})g^{\frac{1}{n}}\omega^{n} \leqslant (\beta + dd^{c}u)\wedge\omega^{n-1}
\]
\end{proposition}

\begin{proof}

\noindent
\begin{align*}
    u \in \mathrm{PSH}_{n-1}(M,\beta,\omega) & \Longrightarrow \beta + \frac{1}{n-1}\Big((\Delta_{\omega}u)\omega - dd^{c}u \Big)\geqslant 0, \\
     & \overset{(\mathrm{Trace})}{\Longrightarrow} \beta \wedge \omega^{n-1} + \frac{1}{n-1}\left((n\frac{dd^{c}u \wedge \omega^{n-1}}{\omega^{n}})\omega^{n} - dd^{c}u \wedge \omega^{n-1}\right)\geqslant 0, \\
     & \Longrightarrow u \in \mathrm{SH}_{1}(M,\beta,\omega).
\end{align*}

Now applying the AM-GM inequality to the eigenvalues of $\alpha_{u}$ we get:

\noindent
\begin{align*}
(\beta+ dd^{c}u) \wedge \omega^{n-1} & \geqslant n\left[\left(\beta + \frac{1}{n-1}\left((n\frac{dd^{c}u \wedge \omega^{n-1}}{\omega^{n}})\omega - dd^{c}u\right)\right)^{n}\right]^{\frac{1}{n}}, \\
 &\geqslant (n b^{\frac{1}{n}})g^{\frac{1}{n}}\omega^{n}.
\end{align*}
\end{proof}

\begin{proposition}\label{prop: MA MA1}
    Let $v\in \mathrm{PSH}(M,\omega) \cap C^{2}(M)$.
    Then $v \in \mathrm{PSH}_{n-1}(M,\omega)$
    and $\omega_{v}^{n} \leqslant \alpha_{v}^{n}$.
\end{proposition}

\begin{proof}

\noindent
\begin{align*}
    \omega + \left.\frac{1}{n-1}\left((n\frac{dd^{c}v \wedge \omega^{n-1}}{\omega^{n}})\omega - dd^{c}v\right)\right. & = \omega + \frac{1}{n-1}\left((n\frac{(dd^{c}v + \omega) \wedge \omega^{n-1}}{\omega^{n}})\omega -n\omega - dd^{c}v\right),\\
    & = \left.\frac{1}{n-1}\left((n\frac{(dd^{c}v + \omega) \wedge \omega^{n-1}}{\omega^{n}})\omega -(\omega + dd^{c}v)\right)\right., \\
    & = \left.\frac{1}{n-1} \left(\mathrm{tr}_{\omega}(\omega_{v})\omega -(\omega_{v})\right) \right.
\end{align*}

    Since $v\in \mathrm{PSH}(M,\omega)$ then we conclude that $v \in \mathrm{PSH}_{n-1}(M,\omega)$ from the last computation.

Now, from the description through eigenvalues of $dd^{c}v$
we have from $v \in \mathrm{PSH}(M,\omega)$:

\noindent
\[
1+\lambda_{i} \geqslant 0 \quad \text{and} \quad  \frac{n-1 + \widehat{\lambda}_{i}}{n-1} = \frac{\sum_{j \neq i}(1+\lambda_{j})}{n-1} \geqslant 0,\forall i \in \{1,\ldots,n\}
\]

Applying AM-GM to each term above we have:

\noindent
\[
\frac{\sum_{j \neq i}(1+\lambda_{j})}{n-1} \geqslant \left[(\prod_{j \neq i}1+\lambda_{j})\right]^{\frac{1}{n-1}}, \forall i \in \{ 1, \ldots , n \}
\]

By taking the product of the inequalities for all $i \in \{ 1, \ldots n \}$ we get:

\noindent
\[
\prod_{i=1}^{n}\left( 1 +\frac{\widehat{\lambda_{i}}}{n-1} \right) \geqslant \prod_{i=1}^{n} \left( 1+\lambda_{i} \right)
\]

\noindent
which is exactly what we wanted to prove in terms of eigenvalues.
\end{proof}

\smallskip

For the rest of this section, our equation, for a unkown $u \in C^{2}(M) \cap \mathrm{PSH}_{n-1}(M,\rho_{\beta}\beta,\omega)$ normalized by $\sup_{M}u =0$,  will be on a compact hermitian manifold $M$:

\noindent
\begin{equation}\label{eq: n-1MA w Gauduchon}
    MA_{n-1}(\rho_{\beta}\beta, \omega, u) = MA_{n-1}(u) := \frac{\left(\rho_{\beta}\beta + \left(\frac{(\Delta_{\omega}u)\omega - dd^{c}u}{n-1}\right)\right)^{n}}{\omega^{n}} = cf,
\end{equation}

\noindent
for $(M, \beta, \omega)$ a compact hermitian manifold with two hermitian metrics, 
$1 \leqslant \rho_{\beta}$ the normalized Gauduchon factor of $\beta$ by~\cite{Gau77}
and $f \in C^{\infty}(M)$.

Our geometric constants under consideration are:

\begin{itemize}

    \item $\mathrm{Vol}_{\omega}(M)$ the volume of $M$ with respect to a metric $\omega$,
    
    \item $1 \leqslant G_{\omega}$ is the uniform upper bound on the Gauduchon factor $\rho_{\omega}$ of a metric $\omega$ from~\cite[Theorem A]{Pan22},
    
    \item $1 < C_{\beta, \omega}$ is the uniform constant such that $C_{\beta,\omega}^{-1}\beta \leqslant \omega \leqslant C_{\beta,\omega}\beta$,
    
    \item $0 < A_{\alpha}, \alpha$ are the uniform constants from Proposition~\ref{prop:global_Skoda},
    
    \item $0 < v_{m_{0}}^{-}(\omega)$ is the uniform constant defined as, given $0 < m_{0}$ uniform constant,
    \[
    v_{m_{0}}^{-}(\omega) := \inf \left\{ \int_{M}(\omega + dd^{c}u)^{n} ~ : ~ u \in \mathrm{PSH}(M,\omega) ~ \text{with} ~ -m_{0} \leqslant u \leqslant 0 \right\}.
    \]

    \noindent
    Such a constant is positive since $\omega$ is hermitian, by~\cite[Proposition 3.4]{GL22}.
    
\end{itemize}

\subsection{\texorpdfstring{$L^{p}$}{Lp} bound on the normalizing constant}

The following result adapts the ideas from~\cite[Lemma 5.9]{KN15},~\cite[Lemma 2.7]{Pan23} and~\cite[Theorem A]{GL25H} we obtain:

\begin{proposition}\label{prop: bound c1}
The normalizing constant $c$ for~\eqref{eq: n-1MA w Gauduchon} 
admits the two bounds:

\noindent
\[
C^{-1}_{\beta, \omega} b \leqslant c \leqslant \left( \frac{C_{Lap}}{n ||f^{\frac{1}{n}}||_{L^{1}(M,\omega^{n})}} \right)^{n}
\]

\noindent
where $b = b(f,C_{\beta,\omega}^{-1}\omega, A_{\alpha})>0$ is the normalizing constant of the Monge-Ampère equation $(C_{\beta,\omega}^{-1}\omega + dd^c u)^{n} = bf\omega^{n}$
and $C_{Lap}:= \operatorname{Vol}_{\omega}(M) G_{\omega}^{n-1} G_{\beta} C_{\beta,\omega}$.
\end{proposition}

See~\cite[Section 2 and 3]{Pan23} for a precise description of the Monge-Ampère constants.

\begin{proof}
\begin{itemize}
    \item[] \textit{Bounding the constant from above}.
\end{itemize}

By~\cite[Theorem A]{Pan22}, there exists a uniform constant $G_{\omega}>1$ such that the normalized Gauduchon factor $\rho_{\omega}$ with respect to $(M,\omega)$ is bounded between 1 and $G_{\omega}$. 
For all smooth $u \in \mathrm{PSH}_{n-1}(M, \rho_{\beta}\beta, \omega)$, we calculate

\noindent
\begin{align*}
    \int_{M}(\rho_{\beta}\beta + dd^{c}u) \wedge \omega^{n-1} & \leqslant \int_{M}(\rho_{\beta}\beta + dd^{c}u) \wedge  {(\rho_{\omega}\omega)}^{n-1} \\
        &\leqslant \int_{M}\rho_{\omega}^{n-1}\rho_{\beta}\beta \wedge \omega^{n-1} + \int_{M}u dd^{c}(\rho_{\omega}\omega)^{n-1} \leqslant G_{\omega}^{n-1} G_{\beta} C_{\beta,\omega}\operatorname{Vol}_{\omega}(M)
\end{align*}

Proposition~\ref{prop: MA Lap1} now yields

\noindent
\[
nc^{\frac{1}{n}}f^{\frac{1}{n}}\omega^{n} \leqslant (\rho_{\beta}\beta + dd^{c}u)\wedge\omega^{n-1}
\]

Integrating we can use the upper bound on the mass of the Hessian operator to conclude:

\noindent
\[
nc^{\frac{1}{n}}||f^{\frac{1}{n}}||_{L^{1}(M, \omega^{n})} \leqslant C_{Lap}, \quad \text{hence} \quad c \leqslant \left( \frac{C_{Lap}}{n||f^{\frac{1}{n}}||_{L^{1}(M, \omega^{n})}} \right)^{n}.
\]

\begin{itemize}
    \item[] \textit{Bounding the constant from below}.
\end{itemize}

We use the existence of solution to the complex Monge-Ampère equation on Hermitian manifolds
that is~\cite[Corollary 1]{TW10b} or~\cite[Theorem B]{GL23} for the case with two different metrics.
Then there exists a pair $(v,b) \in \mathrm{PSH}(M,C_{\beta,\omega}^{-1}\omega)\cap C^{\infty}(M) \times \mathbb{R}_{>0}$
such that:

\noindent
\[
(C_{\beta,\omega}^{-1}\omega + dd^{c}v)^{n} = bf\omega^{n}, \quad \sup_{M}v = 0
\]

Now we use Proposition~\ref{prop: MA MA1} to obtain the comparison:

\noindent
\[
bc^{-1} MA_{n-1}(\beta, \omega, u) \leqslant bf  \leqslant  MA_{n-1}(C_{\beta, \omega}^{-1}\omega, C_{\beta, \omega}^{-1}\omega, v) \leqslant C_{\beta, \omega} MA_{n-1}(\beta, \omega, v)
\]

Then, by the domination principle (Lemma~\ref{lema: DP1}) for $v$ and $u$ we get

\noindent
\[
C^{-1}_{\beta, \omega} b c^{-1} \alpha_{u}^n \leqslant \alpha_{v}^n \Rightarrow C_{\beta, \omega}^{-1} b c^{-1} \leqslant 1. \Rightarrow c \geqslant C_{\beta, \omega}^{-1} b. 
\]

By~\cite[Theorem B]{GL23} the constant $b$ is bounded from below uniformly, hence so is $c$. 
\end{proof}

For the~\ref{set: sing family setting},
by~\cite{Pan22,Pan23}, the constants in Proposition~\ref{prop: bound c1} are under control.

\subsection{\texorpdfstring{$L^{\infty}$}{Linfty} a priori estimate}

Adapting the technique from~\cite[Theorem B]{GL25} we prove:

\begin{theorem}\label{teo: C0 Vincent}

For $(M,\beta, \omega)$ a compact hermitian manifold with two hermitian metrics, as described above for~\eqref{eq: n-1MA w Gauduchon}.
Then, we have the uniform bounds:

\begin{itemize}
    \item[(A):] $c  \leqslant C_{0}$, 
    
    for $C_{0} = \left( \frac{C_{Lap}}{n ||f^{\frac{1}{n}}||_{L^{1}(M,\omega^{n})}} \right)^{n} > 0$, where $C_{Lap}:= \operatorname{Vol}_{\omega}(M) G_{\omega}^{n-1} G_{\beta} C_{\beta,\omega}$

    \item[(B):] $-C \leqslant u  \leqslant 0$, 
    
    for $C = C(C^{-1}_{\beta, \omega}, ||f||_{L^{p}(M, \omega^{n})}, n, p, ||u||_{L^{1}(M,\omega^{n})}, \frac{C_{0}}{v^{-}_{m_{0}}(C^{-1}_{\beta, \omega}\omega)}) > 0$ and a $m_0 > 0$ uniformly depending on $f,C_{0}, C_{\beta, \omega}$.

    \item[(C):] $c_{0} \leqslant c$, 
    
    for $c_{0} = \frac{v^{-}_{m_{0}}(C^{-1}_{\beta, \omega}\omega)}{||f||_{L^{1}(M,\omega^{n})}} > 0$.
\end{itemize}
    
\end{theorem}

\begin{proof}

The upper bound on the constant (A) follows exactly as in Proposition~\ref{prop: bound c1}

\begin{itemize}
    \item[] \textit{Bounding the solution from below}.
\end{itemize}

To get the second bound $(B)$ we use the comparison with the 
complex Monge-Ampère through a psh envelope, that is:

Firstly, to aliviate the notation we will substitute $\rho_{\beta}$ for $\rho$ and set $\widetilde{\omega} := C^{-1}_{\beta, \omega}\omega \leqslant \beta$ by~\eqref{cst: geo constant}.
Now consider $\psi  := \mathrm{P}_{\widetilde{\omega}}(u)$ the $\widetilde{\omega} $-psh envelope of $u$.
The function $\psi$ is $C^{1,1}$ and by~\cite[Theorem 2.3]{GL22} the complex Monge-Ampère 
operator of the envelope is supported on the contact set, i.e., 
$MA_{\widetilde{\omega}}(\psi )$ is concentrated on $\mathcal{C} := \{ \psi  = u  \}$.

It is known that on the contact set $\mathcal{C}$ the 
function $u $ is $\widetilde{\omega} $-psh. 
Then on $\mathcal{C}$ the operator $MA_{\widetilde{\omega} }(u )$ is the complex Monge-Ampère of a $\widetilde{\omega} $-psh function.

Since the inequality on Proposition~\ref{prop: MA MA1} is a pointwise comparison we have:

\noindent
\[
MA_{\widetilde{\omega} }(u )|_{\mathcal{C}} \leqslant  MA_{n-1}(\widetilde{\omega}, \widetilde{\omega}, u)|_{\mathcal{C}} \leqslant C_{\beta, \omega} MA_{n-1}(\beta, \omega, u )|_{\mathcal{C}} \leqslant C_{\beta, \omega} MA_{n-1}(\rho \beta, \omega, u )|_{\mathcal{C}}
\]

\noindent
at all points of $\mathcal{C}$. Since $u  \in \mathrm{PSH}_{n-1}(M,\rho \beta, \omega )$
we get:

\noindent
\[
MA_{\widetilde{\omega} }(\psi ) = \mathbb{1}_{\mathcal{C}} MA_{\widetilde{\omega} }(u ) \leqslant C_{\beta, \omega} \mathbb{1}_{\mathcal{C}} MA_{n-1}(\rho \beta, \omega, u)  \leqslant C_{\beta, \omega} MA_{n-1}(\rho \beta, \omega, u).
\]

Then, we get:

\noindent
\[
MA_{\widetilde{\omega} }(\psi ) \leqslant C_{\beta, \omega} MA_{n-1}(\rho \beta, \omega, u) = C_{\beta, \omega} cf \leqslant C_{\beta, \omega} C_{0}f,
\]

\noindent
where the last inequality follows from the bound $(A)$ in the previous step.
For a uniformly controlled constant $m_{0}>0$ we have $Osc_{M}(\psi ) \leqslant m_{0}$, by~\cite[Theorem A]{GL23}.
Since we have $\psi  \leqslant u  \leqslant 0$ and the oscillation is under control we have:

\noindent
\[
\int_{M}MA_{\widetilde{\omega} }(\psi ) \widetilde{\omega}^{n} \geqslant v^{-}_{m_{0}}(\widetilde{\omega})
\]

\noindent
which is under control, since $m_{0}$ is under control by~\cite[Proposition 3.4]{GL22}.
We compute:

\noindent
\begin{align*}
    0 \leqslant (- \underset{M}{sup}(\psi ))^{\delta}v^{-}_{m_{0}}(\widetilde{\omega} ) & \leqslant \int_{M}(-\psi )^{\delta}MA_{\widetilde{\omega} }(\psi )\widetilde{\omega}^{n}, \\
    & = \int_{\mathcal{C}}(-\psi )^{\delta}MA_{\widetilde{\omega} }(\psi )\widetilde{\omega}^{n} = \int_{M}(-u )^{\delta}MA_{\widetilde{\omega} }(\psi )\widetilde{\omega}^{n}, \\
    & \leqslant C_{0}\int_{M}(-u )^{\delta}f\widetilde{\omega}^{n} \leqslant C_{0} C^{-n}_{\beta, \omega}||f||_{L^{p}(M, \omega^{n})} ||u ||_{L^{1}(M,\omega^{n})}^{1/q}.
\end{align*}

\noindent
where the last inequality follows from the Hölder inequality for
$\delta = \frac{1}{q}$. 
The term $||u ||_{L^{1}(M,\omega^{n})}$ is under control by the 
compactness in $L^{1}$ of $\mathrm{SH}_{1}(M, \beta, \omega)$ see~\cite[Lemma 8]{GP24}.

This implies that $\psi $ is uniformly bounded.
Since, by definition $\psi  \leqslant u $ we conclude:

\noindent
\[
-u  \leqslant m_{0} + \left( \frac{C_{0} C^{-n}_{\beta, \omega} ||f||_{L^{p}(M, \omega^{n})} ||u ||_{L^{1}(M,\omega^{n})}^{1/q}}{v^{-}_{m_{0}}(\widetilde{\omega})} \right)^{q}.
\]

\begin{itemize}
    \item[] \textit{Bounding the constant from below}.
\end{itemize}

To prove the last bound $(C)$ one can just compare the Monge-Ampère masses:

\noindent
\[
MA_{\widetilde{\omega}}(\psi ) \leqslant cf \Longrightarrow v^{-}_{m_{0}}(\widetilde{\omega}) \leqslant c \int_{M}f \widetilde{\omega}^{n} \Longrightarrow \frac{v^{-}_{m_{0}}(\widetilde{\omega})}{C^{-n}_{\beta, \omega} ||f||_{L^{1}(M,\omega^{n})}} \leqslant c
\]

\noindent
by integrating with respect to $\widetilde{\omega}$ and since $\psi $ is uniformly bounded.
\end{proof}

\subsection{Proof of \texorpdfstring{\ref{thm: C0 family}}{Theorem A}}

In~\ref{thm: C0 family},
a $L^{\infty}$ estimate for the solution $u_{t}$ of the family of complex $(n-1)$-Monge-Ampère equations,  
the following constant are under control:

\begin{enumerate}
    \item $||f_{t}^{\frac{1}{n}}||_{L^{1}(X_t,\omega_{t}^{n})}, ||f_{t}||_{L^{p}(X_{t},\omega_{t}^{n})}$ are under control by assumption of the density bounds~\eqref{eq: density bounds}.
    
    \item $G_{\omega},G_{\beta}$ are under control since $\rho_{t}$ satisfy the Gauduchon bound~\eqref{eq: Gauduchon bound}.
    
    \item $\mathrm{Vol}_{\omega_{t}}(X_{t})$ is under control by the volume bound~\eqref{eq: volume bound}.
    
    \item $v_{m_{0}}^{-}(\omega)$ is uniform, for a $m_{0}$ uniform, by~\cite[Proposition 3.4]{GL22}.
    
\end{enumerate}

Hence, Theorem~\ref{teo: C0 Vincent} completes the proof of~\ref{thm: C0 family}.

\smallskip

These bounds above are satisfied in~\ref{set: sing family setting}
when $X_{0}$ has canonical singularities, since:

\begin{enumerate}
    \item By~\cite[Lemma 4.4]{DNGG23} the adapted volume satisfies~\eqref{eq: density bounds}.

    \item By~\cite[Theorem A]{Pan22}, the Gauduchon factors of any smoothable variety satisfy~\eqref{eq: Gauduchon bound}.
    
    \item By~\cite[Section 1]{Pan22}, the volume of any smoothable variety satisfies~\eqref{eq: volume bound}.
    
\end{enumerate}

\section{\texorpdfstring{$C^{2}$}{C2} Estimate on holomorphic families}\label{CYG: C2}

From here on we will work in a holomorphic family $\pi: (\mathcal{M}, \beta, \omega) \to \mathbb{D}$ in~\ref{set: sing family setting},
without singular fibers, that is the central fiber $M_{0} = \pi^{-1}(0)$ is a complex manifold.
We will further restrict ourself fixing metric $\omega$ being a Kähler metric, since this is the setting of the geometric consequence of~\ref{teo: estimates HF}.

In this section we will prove a $C^{2}$-estimate that depends on the norm of the gradient as in~\cite{HMW10}, who's original strategy dates back to~\cite{CW01}.
Our proof is a more precise description of~\cite[Theorem 4.1]{TW17} that will make the estimate uniform in families.
To help the reader we will align our notation with theirs and reference their paper for the computations that we will omit here.

\subsection{Uniform constants}

Firstly we highlight some of the uniform constants that the estimate will depend on
fixing their notation. 
Recall that for this setting $M$ is a Kähler manifold, $\omega$ a Kähler metric, however $\beta$ is Hermitian.

\textbf{Metric comparison constant} Let, as in~\eqref{cst: geo constant}, $C_{\beta,\omega} >1$ such that:

\noindent
\[
C_{\beta,\omega}^{-1} \beta \leqslant \omega \leqslant C_{\beta,\omega} \beta
\]

\textbf{Bisectional curvature bound:} Let $B>0$ constant such that:

\noindent
\[
-B | \xi_{1} |^{2} | \xi_{2} |^{2} \leqslant R_{i,\bar{\jmath},k,\bar{l}} \xi^{i}_{1} \overline{\xi^{j}_{1}} \xi^{k}_{2} \overline{\xi^{l}_{2}}  \leqslant B | \xi_{1} |^{2} |  \xi_{2} |^{2}, \quad \xi_{1},\xi_{2} \in \mathbb{C}^{n}.
\]

\textbf{$L^{\infty}$-estimate:} Let $C_{L^{\infty}} >0$ be the uniform constant from Theorem~\ref{teo: C0 Vincent}
that is:

\noindent
\[
c + c^{-1} + \sup_{M} |u| \leqslant C_{L^{\infty}}.
\]

\textbf{Notation:} To use a similar notation to~\cite[Section 4]{TW17} we set the following parallel between the hermitian forms and their riemannian metrics:

\noindent
\[
\alpha_{u} = \alpha \leadsto \widetilde{g}, \quad \omega \leadsto g, \quad \beta \leadsto h. \text{ Hence  } \widetilde{g}_{i\bar{\jmath}} = h_{i\bar{\jmath}} + \frac{1}{n-1}\left(  (\Delta u)g_{i\bar{\jmath}} - u_{i\bar{\jmath}} \right).
\]

\noindent
and

\noindent
\begin{equation}\label{eq: n-1 MA HF}
    \left( \beta + \frac{(\Delta_{\omega}u)\omega - dd^{c}u}{n-1} \right)^{n} = f\omega^{n},
\end{equation}

\noindent
for $\sup_{M}u = 0$ and $0 < f \in C^{\infty}(M)$.\footnote{We are including the normalizing constant inside the density $f$ since we already have a $L^{\infty}$-estimate.}

\subsection{Uniform \texorpdfstring{$C^{2}$}{C2} estimate on manifolds}

We have the following $C^{2}$-estimate:

\begin{theorem}\label{teo: C2 estimate hol fam}
For $(M,\beta, \omega)$ a compact Kähler manifold, where $\omega$ is a Kähler metric and $\beta$ is a hermitian metric.
Then, there is a uniform constant $C>0$ such that:

\noindent
\[
|| dd^{c}u ||_{L^{\infty}} \leqslant C (\sup_{M}| \nabla u |_{g}^{2} + 1)
\]

\noindent
for $C = C(B, C_{L^{\infty}}, ||h||_{C^{2}}, ||\log(f)||_{C^{2}}, ||f||_{L^{\infty}}, n, C_{\beta, \omega} )$

\end{theorem}

\begin{proof}

    As in~\cite[Equation (4.2)]{TW17} we use the tensor $\eta$:

\noindent
    \begin{equation}\label{eq: tensor eta}
        \eta_{i\bar{\jmath}} = u_{i\bar{\jmath}} - (n-1)h_{i\bar{\jmath}} + (tr_{g}h)g_{i\bar{\jmath}}  = (tr_{g}\widetilde{g})g_{i\bar{\jmath}} - (n-1)\widetilde{g}_{i\bar{\jmath}}; \quad \widetilde{g}_{i\bar{\jmath}} = \frac{1}{n-1}((tr_{g}\widetilde{g})g_{i\bar{\jmath}} - \eta_{i\bar{\jmath}})   
    \end{equation}

    We will apply the maximum principle to\footnote{Here our notation slightly differs from~\cite{TW17}.}: 
    
    \noindent
    \[
    H(x,\xi) = \log(\eta_{i\bar{\jmath}} \xi^{i} \overline{\xi^{j}}) + \psi_{1}(|\nabla u|^{2}_{g}) + \psi_{2}(u),
    \]

    \noindent
    for $x \in M, \xi \in T^{1,0}_{x}M$ a $g$-unit vector and functions $\psi_{1},\psi_{2}$ defined as

    \noindent
    \begin{align}\label{eq: test function 1}
        & \psi_{1}(s) = -\frac{1}{2}\log \left( 1 - \frac{s}{2K} \right), \text{ for } 0 \leqslant s \leqslant K-1, \text{ for } K = \sup_{M} | \nabla u |_{g}^{2} + 1. \\
        & \psi_{2}(t) = -A\log \left( 1 + \frac{t}{2L} \right), \text{ for } -L + 1 \leqslant t \leqslant 0, \text{ for } L = C_{L^{\infty}} + 1, A = 2L(C_{0} +1).
        \label{eq: test function 2}
    \end{align}

\noindent
for a $C_{0} >0$ uniform to be chosen later.

We treat $H$ on a  subcompact $W$ of the $g$-unit tangent bundle of $M$
where $\eta_{i\bar{\jmath}} \xi^{i} \overline{\xi^{j}} \geqslant 0$,
with $H = -\infty$ when $\eta_{i\bar{\jmath}} \xi^{i} \overline{\xi^{j}} = 0$.
Hence its maximum is achieved in $\mathring{W}$ a point $(x_{0}, \xi_{0})$.
Pick holomorphic coordinates as in~\cite[page 324]{TW17}, normal with respect to $g$, at $x_{0}$ such that:

\noindent
\begin{align}\label{eq: g e eta coord}
& g_{i\bar{\jmath}}  = \delta_{ij}, \quad \eta_{i\bar{\jmath}} = \eta_{i}\delta_{ij}, \quad \eta_{1} \geqslant \eta_{2} \geqslant \cdots \geqslant \eta_{n}. \\
& \widetilde{g}_{i\bar{\jmath}} = \lambda_{i}\delta_{ij}, \quad \eta_{i} = \sum_{k=1}^{n} \lambda_{k} - (n-1)\lambda_{i}, \quad 0 < \lambda_{1} \leqslant \lambda_{2} \leqslant \cdots \leqslant \lambda_{n}.
    \label{eq: g til coord}
\end{align}

We have that the quantities $\lambda_{n}, \eta_{1}, tr_{\omega}\alpha$ are uniformily equivalent:

\noindent
\begin{equation}\label{eq: quant equiv}
\frac{1}{n} tr_{g}\widetilde{g} \leqslant \lambda_{n} \leqslant \eta_{1} \leqslant (n-1)\lambda_{n} \leqslant (n-1) tr_{g}\widetilde{g}.
\end{equation}

As $\eta_{1}$ is the largest eigenvalue of $(\eta_{i\bar{\jmath}})$ the maximum of $H$ is achieved for the vector $\xi_{0} = \frac{\partial}{\partial z^{1}}$
that can be locally extended to $\xi_{0} = g_{1\bar{1}}^{-\frac{1}{2}}\frac{\partial}{\partial z^{1}}$.
Finally we set:

\noindent
\begin{equation}\label{eq: functional Q}
    Q(x) = H(x, \xi_{0}) = \log \left( g_{1\bar{1}}^{-1}\eta_{1} \right) + \psi_{1}(| \nabla u |_{g}^{2}) + \psi_{2}(u).
\end{equation}

The maximum of $Q$ is achieved at $x_{0}$.
We will prove that $\eta_{1}(x_{0}) \leqslant CK$, finishing the proof.

We may assume without loss of generality, $\eta_{1} >1$ and $u_{1\bar{1}}>0$ at $x_{0}$. We define the tensor $\Theta^{i\bar{\jmath}}$, and its linear operator, as in~\cite[page 324]{TW17}:

\noindent
\begin{equation}\label{eq: tensor theta}
    \Theta^{i\bar{\jmath}} =  \frac{1}{n-1}\left( (tr_{g}\widetilde{g})g^{i\bar{\jmath}} - \widetilde{g}^{i\bar{\jmath}} \right) >0, \quad L(v) = \Theta^{i\bar{\jmath}} v_{i\bar{\jmath}} = \frac{1}{n-1} \left( (tr_{g}\widetilde{g})\Delta v - \widetilde{\Delta} v \right).
\end{equation}

By the choice of $g$-normal coordinates we get $\nabla \xi_{0} (x_{0}) = 0$, for the covariant derivative, giving us the identities at $x_{0}$, analogous to~\cite[(4.7) and (4.9)]{TW17}:

\noindent
\begin{equation}\label{eq: Qi}
    0 = Q_{i} = \frac{\eta_{1i}}{\eta_{1}} + \psi_{1}^{\prime} \cdot \left( \sum_{p}u_{p}u_{\bar{p}i} + \sum_{p}u_{pi}u_{\bar{p}} \right) + \psi_{2}^{\prime} \cdot u_i .
\end{equation}

Since $\Theta^{i\bar{\jmath}}$ is diagonal at $x_{0}$, and writing $(\xi^{i})$ for the coordinates of $\xi_{0}$, we obtain:

\noindent
\begin{equation}
\begin{split}\label{eq: L(Q)}
    L(Q) = & \sum_i \frac{\Theta^{i\bar{\imath}} \eta_{1 i\bar{\imath}}}{\eta_{1}} - \sum_i \frac{\Theta^{i\bar{\imath}} | \eta_{1i}|^2}{(\eta_{1})^2}
+ \sum_i\Theta^{i\bar{\imath}}(\xi^1_{i\bar{\imath}} + \overline{\xi^1_{\bar{\imath} i}}) + \psi_{2}^{\prime} \sum_i \Theta^{i\bar{\imath}} u_{i\bar{\imath}} \\
& + \psi_{2}^{\prime \prime} \sum_i \Theta^{i\bar{\imath}} |u_i|^2 + \psi_{1}^{\prime \prime} \sum_i \Theta^{i\bar{\imath}} \left| \sum_p u_p u_{\bar{p} i} + \sum_p u_{pi} u_{\bar{p}} \right|^2 \\
& + \psi_{1}^{\prime} \sum_{i,p} \Theta^{i\bar{\imath}} |u_{p \bar{\imath}}|^2  + \psi_{2}^{\prime} \sum_{i,p} \Theta^{i\bar{\imath}} |u_{pi}|^2 + \psi_{1}^{\prime} \sum_{i,p} \Theta^{i\bar{\imath}} (u_{pi\bar{\imath}} u_{\bar{p}} + u_{\bar{p} i \bar{\imath}} u_p),
\end{split}
\end{equation}

After tricky direct computations, differentiating covariantly the $(n-1)$Monge-Ampère equation
and commuting derivaties as in~\cite[page 552]{HMW10} we get, since $L(Q) \leqslant 0$, at $x_{0}$:

\noindent
\begin{equation}
\begin{split}\label{ineq: main ineq}
0 \geqslant & \frac{\sum_{i,j} \widetilde{g}^{i\bar{\imath}} \widetilde{g}^{j\bar{\jmath}} ( g_{j\bar{\imath}} \sum_a u_{a\bar{a} \bar{1}} - u_{j\bar{\imath} \bar{1}}+ \widehat{h}_{j\bar{\imath}\bar{1}}) (g_{i\bar{\jmath}} \sum_b u_{b\bar{b} 1} - u_{i\bar{\jmath}1} + \widehat{h}_{i\bar{\jmath}1})}{{(n-1)^2\eta_{1}}}   - \sum_i \frac{\Theta^{i\bar{\imath}} | \eta_{1i}|^2}{(\eta_{1})^2} \\
& + \frac{1}{\eta_{1}} \left(  F_{1\bar{1}} - \sum_{i} \Theta^{i\bar{\imath}} \left( \sum_a u_{a\bar{\imath}} R_{i \bar{a} 1 \bar{1}} - \sum_a u_{a\bar{1}} R_{i \bar{a} 1 \bar{\imath}}\right)  \right. \\
& \left. - \sum_i \widetilde{g}^{i\bar{\imath}} h_{i\bar{\imath}1\bar{1}} - \sum_i \Theta^{i\bar{\imath}} \left(  \widehat{h}_{1\bar{1} i\bar{\imath}} - (tr_{g}{h})_{i\bar{\imath}} \right) \right) \\
& + \psi_{2}^{\prime} \sum_i \Theta^{i\bar{\imath}} u_{i\bar{\imath}}
+ \psi_{2}^{\prime \prime} \sum_i \Theta^{i\bar{\imath}} |u_i|^2 + \psi_{1}^{\prime \prime} \sum_i \Theta^{i\bar{\imath}} \left| \sum_p u_p u_{\bar{p} i} + \sum_p u_{pi} u_{\bar{p}} \right|^2 \\
&+ \psi_{1}^{\prime} \sum_{i,p} \Theta^{i\bar{\imath}} |u_{p \bar{\imath}}|^2  + \psi_{1}^{\prime} \sum_{i,p} \Theta^{i\bar{\imath}} |u_{pi}|^2 + \psi_{1}^{\prime} \sum_{i,p} \Theta^{i\bar{\imath}} (u_{pi\bar{\imath}} u_{\bar{p}} + u_{\bar{p} i \bar{\imath}} u_p),
\end{split}
\end{equation}

\noindent
which coresponds to~\cite[(4.16)]{TW17}, and all the intermediate steps are the same,
for $e^{F} = f$ and $\widehat{h}_{i\bar{\jmath}} = (n-1)h_{i\bar{\jmath}}$.
Since all the intermediate computations will be the same as in~\cite{TW17}
we skip them focusing only on describing the constants and inequalities more precisely:

\textbf{In~\cite[Equation (4.17)]{TW17}:} at $x_{0}$

\noindent
\begin{equation}\label{eq: 4.17}
    \sum \Theta^{i\bar{\imath}} = tr_{\widetilde{g}}g, \quad \text{and} \quad |u_{i\bar{\jmath}}| = | \eta_{i\bar{\jmath}} + (n-1)h_{i\bar{\jmath}} - (tr_{g}h)g_{i\bar{\jmath}}| \leqslant C_{1}\eta_{1},
\end{equation}

\noindent
for $C_{1} = 1 + (2n-1)||h||_{L^{\infty}}$, by~\eqref{eq: tensor eta}
and since $\eta_{1} >1$ is the highest eigenvalue of $(\eta_{i\bar{\jmath}})$ at $x_{0}$.

Now, we move on to bounding from below the second and third line of~\eqref{ineq: main ineq}

\textbf{In~\cite[Equations (4.18) and (4.19)]{TW17}:}

\noindent
\begin{equation} \label{eq: curvature term}
\begin{split}
& \frac{1}{\eta_{1}} \left(  F_{1\bar{1}} - \sum_{i} \Theta^{i\bar{\imath}}  \left( \sum_a u_{a\bar{\imath}} R_{i \bar{a} 1 \bar{1}} - \sum_a u_{a\bar{1}} R_{i \bar{a} 1 \bar{\imath}}\right) \right.  \\ 
& \left. - \sum_i \widetilde{g}^{i\bar{\imath}} h_{i\bar{\imath}1\bar{1}} - \sum_i \Theta^{i\bar{\imath}} \left( \widehat{h}_{1\bar{1} i\bar{\imath}} - (tr_{g}{h})_{i\bar{\imath}} \right) \right) \geqslant  - C_{2} tr_{\widetilde{g}}{g} - C_{3},
\end{split}
\end{equation}

\noindent
for $C_{2} = 2n(||h||_{C^{2}} +  C_{1}B)$ and $C_{3} = ||F||_{C^{2}}$. Since $\eta_{1}>1$ and using equations~\eqref{eq: 4.17}
with the bound on the bisectional curvature $B$. And by taking the $\widetilde{g}$-trace of $\widetilde{g}$ we get:

\noindent
\begin{equation} \label{eq: Thetauii}
\sum_i \Theta^{i\bar{\imath}} u_{i\bar{\imath}} = \frac{1}{n-1} ( (tr_{\widetilde{g}}{g}) \Delta u - \widetilde{\Delta} u) = n - tr_{\widetilde{g}}{h},
\end{equation}

Now we will bound the last term written in~\eqref{ineq: main ineq}:

\textbf{In~\cite[Equations (4.20) and (4.23)]{TW17}:}

Again by using the covariant differentiation of the PDE
and commuting derivatives:

\noindent
\begin{equation}\label{eq: third order term}
\begin{split}
\lefteqn{ \psi_{1}^{\prime} \sum_{i,p} \Theta^{i\bar{\imath}} (u_{pi\bar{\imath}} u_{\bar{p}} + u_{\bar{p} i\bar{\imath}} u_p )} \\={} & \psi_{1}^{\prime} \sum_p ( F_p u_{\bar{p}} + F_{\bar{p}} u_p)
- \psi_{1}^{\prime} \sum_{i,p} \widetilde{g}^{i\bar{\imath}} (h_{i\bar{\imath}p} u_{\bar{p}} + h_{i\bar{\imath} \bar{p}} u_p)
\\
& + \psi_{1}^{\prime} \sum_{i,p,q} \Theta^{i\bar{\imath}} u_q u_{\bar{p}} R_{p\bar{q} i\bar{\imath}}  \\
 \geqslant {} & - ||F||_{C^{1}}(1 + | \nabla u|^2_g)  | \psi_{1}^{\prime}| - (||h||_{C^{1}} + n^{2}B)tr_{\widetilde{g}}{g} (| \nabla u|^{2}_{g} + 1)  | \psi_{1}^{\prime}| \\
 \geqslant {} &  - C_{4}tr_{\widetilde{g}}g - C_{5}.
\end{split}
\end{equation}

\noindent
for $C_{4} = \frac{||h||_{C^{1}} + n^{2}B}{2}$ and $C_{5} = \frac{||F||_{C^{1}}}{2}$.
By using the Cauchy-Buniakovski-Schwarz inequality, the bound on the bisectional curvature $B$, $ \frac{1}{2K} \geqslant \psi_{1}^{\prime} \geqslant \frac{1}{4K}$ and~\eqref{eq: 4.17}.

Combining~\eqref{ineq: main ineq} with~\eqref{eq: curvature term},~\eqref{eq: Thetauii} and~\eqref{eq: third order term} yields:

\noindent
\begin{equation} \label{ineq: main ineq 2}
\begin{split}
0 \geqslant {} &  \frac{\sum_{i,j} \widetilde{g}^{i\bar{\imath}} \widetilde{g}^{j\bar{\jmath}} ( g_{j\bar{\imath}} \sum_a u_{a\bar{a} \bar{1}} - u_{j\bar{\imath} \bar{1}}+ \widehat{h}_{j\bar{\imath}\bar{1}}) (g_{i\bar{\jmath}} \sum_b u_{b\bar{b} 1} - u_{i\bar{\jmath}1} + \widehat{h}_{i\bar{\jmath}1})}{{(n-1)^2\eta_{1}}}   - \sum_i \frac{\Theta^{i\bar{\imath}} | \eta_{1i}|^2}{(\eta_{1})^2} \\
& + \psi_{2}^{\prime \prime} \sum_i \Theta^{i\bar{\imath}} |u_i|^2 + \psi_{1}^{\prime \prime} \sum_i \Theta^{i\bar{\imath}} \left| \sum_p u_p u_{\bar{p} i} + \sum_p u_{pi} u_{\bar{p}} \right|^2 + \psi_{1}^{\prime} \sum_{i,p} \Theta^{i\bar{\imath}} |u_{p \bar{\imath}}|^2 \\
& + \psi_{1}^{\prime} \sum_{i,p} \Theta^{i\bar{\imath}} |u_{pi}|^2 - tr_{\widetilde{g}}g (C_{\beta, \omega}\psi_{2}^{\prime}  + C_{2} + C_{4}) - (n \psi_{2}^{\prime} + C_{3} + C_{5}),
\end{split}
\end{equation}

\noindent
since $ -\frac{A}{L} \leqslant \psi_{2}^{\prime} \leqslant -\frac{A}{2L} = -(C_{0} +1 )$.
Then by choosing $C_{0} = \frac{C_{2} + C_{4}}{C_{\beta, \omega}} + 2\frac{n || h ||_{C^{1}}^{2}}{(n-1)C_{\beta, \omega}}$
the only term left to control is $\sum_i \frac{\Theta^{i\bar{\imath}} | \eta_{1i}|^2}{(\eta_{1})^2}$ as the others are positive.
As in~\cite[page 328]{TW17} we define:

\noindent
\[
\delta = \frac{1}{1+ 2A} = \frac{1}{1 + 4L(C_{0} + 1)},
\]

\noindent
which is uniform since so are $L$ and $C_{0}$.
Now we have two cases as in~\cite{HMW10,TW17}.

\smallskip
\textbf{Case 1} (Assume $\lambda_{2} \leqslant (1-\delta)\lambda_{n}$). We use~\eqref{eq: Qi} and~\eqref{eq: 4.17} to obtain

\noindent
\begin{equation}
\begin{split}
    - \sum_{i} \frac{\Theta^{i\bar{\imath}} |\eta_{1\bar{1}i}|^{2}}{(\eta_{1})^{2}} & = - \sum_{i} \Theta^{i\bar{\imath}} \left|\psi_{1}^{\prime}\left( \sum_{p} u_{p} u_{\bar{p}i}  +  \sum_{p} u_{pi} u_{\bar{p}}  \right) + \psi_{2}^{\prime} u_{i} \right|^{2} \\
    & -2(\psi_{1}^{\prime})^{2} \sum_{i} \Theta^{i\bar{\imath}} \left| \sum_{p} u_{p} u_{\bar{p}i} + \sum_{p} u_{pi} u_{\bar{p}} \right|^{2} - 8(C_{0} +1)^{2}K tr_{\widetilde{g}}g,
\end{split}
\end{equation}

\noindent
since $|\psi_{2}^{\prime}| \leqslant \frac{A}{L} = 2(C_{0} + 1)$.
Then, using $\psi_{1}^{\prime \prime} = 2(\psi_{1}^{\prime})^{2}, \psi_{2}^{\prime \prime} > 0, \psi_{1}^{\prime}, -\psi_{2}^{\prime} - C_{0} \geqslant 1, \psi_{1}^{\prime} \geqslant \frac{1}{4K}$
and erasing positive unhelpful terms we obtain from~\eqref{ineq: main ineq 2}

\noindent
\begin{equation} 
0  \geqslant \frac{1}{4K} \Theta^{n\bar{n}}u_{n\bar{n}} - 8(C_{0} + 1)^{2}K tr_{\widetilde{g}}g - C_{6},
\end{equation}

\noindent
for $C_{6} = C_{3} + C_{5}$. Since $\Theta^{n\bar{n}} \geqslant \frac{1}{n} \sum_{i} \Theta^{i \bar{\imath}} \geqslant \frac{1}{n} tr_{\widetilde{g}} g$:

\noindent
\begin{equation} \label{ineq: case 1 final ineq}
\begin{split}
\frac{1}{n} (tr_{\widetilde{g}}g)u_{n\bar{n}}^{2} \leqslant 32(C_{0} + 1)^{2}K^2 tr_{\widetilde{g}}g  + C_{6}K.   
\end{split}
\end{equation}

By the assumption $\lambda_{2} \leqslant (1 - \delta)\lambda_{n}$
and~\eqref{eq: tensor eta}

\noindent
\[
u_{n\bar{n}} \leqslant \lambda_{1} - \delta \lambda_{n} + (2n - 1) ||h||_{L^{\infty}} \leqslant - \frac{\delta}{2}\lambda_{n}
\]

\noindent
since $\delta <1$ and without loss of generality $\lambda_{1} < 1$ and $\frac{\lambda_{n}}{2} \geqslant \frac{(2n - 1) ||h||_{L^{\infty}} + 1}{\delta}$.
Hence $u_{n\bar{n}}^{2} \geqslant \frac{\delta^{2}}{4}\lambda_{n}$, with~\eqref{ineq: case 1 final ineq} one concludes

\noindent
\[
\lambda_{n} \leqslant C_{7}K
\]

for $C_{7} = \frac{8n C_{6}}{\delta^{2}}$, since $1 \leqslant tr_{\widetilde{g}}g \cdot \lambda_{n}$
and without loss of generality $\frac{\lambda_{n}^{2}}{2} \geqslant \frac{32(C_{0} + 1)^{2} K^{2} (4n)}{\delta^{2}}$.

By~\eqref{eq: quant equiv}, this implies the desired estimate $\eta_{1} \leqslant (n-1)C_{7} K$.

\smallskip
\textbf{Case 2} (Assume $\lambda_{2} \geqslant (1-\delta)\lambda_{n}$). Similar to case one, we have for the summand $i=1$

\noindent
\[
- \frac{\Theta^{1\bar{1}}|\eta_{1\bar{1}1}|^{2}}{(\eta_{1})^{2}} \geqslant -2(\psi_{1}^{\prime})^{2}  \Theta^{1\bar{1}} \left| \sum_{p} u_{p} u_{\bar{p}i} + \sum_{p} u_{pi} u_{\bar{p}} \right|^{2} - 8(C_{0} +1)^{2}n K \Theta^{1\bar{1}},
\]

\noindent
combining it with $2(\psi_{1}^{\prime})^{2} = \psi_{1}^{\prime \prime}, \psi_{1}^{\prime} \geqslant \frac{1}{4K}, -C_{\beta, \omega}\psi_{2}^{\prime} - C_{2} - C_{4} \geqslant 2\frac{n || h ||_{C^{1}}^{2}}{(n-1)}$ and~\eqref{ineq: main ineq 2}

\noindent
\begin{equation} \label{ineq: ineq case 2 semi final}
\begin{split}
0 \geqslant {} &  \frac{\sum_{i,j} \widetilde{g}^{i\bar{\imath}} \widetilde{g}^{j\bar{\jmath}} ( g_{j\bar{\imath}} \sum_a u_{a\bar{a} \bar{1}} - u_{j\bar{\imath} \bar{1}}+ \widehat{h}_{j\bar{\imath}\bar{1}}) (g_{i\bar{\jmath}} \sum_b u_{b\bar{b} 1} - u_{i\bar{\jmath}1} + \widehat{h}_{i\bar{\jmath}1})}{{(n-1)^2\eta_{1}}}   - \sum_{i = 2}^{n} \frac{\Theta^{i\bar{\imath}} | \eta_{1i}|^2}{(\eta_{1})^2} \\
& + \psi_{2}^{\prime \prime} \sum_i \Theta^{i\bar{\imath}} |u_i|^2 + \psi_{1}^{\prime \prime} \sum_{i = }^{n} \Theta^{i\bar{\imath}} \left| \sum_p u_p u_{\bar{p} i} + \sum_p u_{pi} u_{\bar{p}} \right|^2 \\
& + \frac{1}{4K} \sum_{i,p} \Theta^{i\bar{\imath}} |u_{pi}|^2 + C_{8} tr_{\widetilde{g}}g - 8(C_{0} + 1)^{2} n K \Theta^{1\bar{1}} - C_{6},
\end{split}
\end{equation}

\noindent
for $C_{8} = 2\frac{n}{n-1}|| h ||_{C^{1}}^{2}$. Without loss of generality, we assume $ \frac{1}{4K} \Theta^{1\bar{1}}u_{1\bar{1}}^{2} \geqslant 8(C_{0} + 1)^{2} n K \Theta^{1\bar{1}}$
then with Lemma~\ref{lema: final bound C2 case 2} below,
which is a quantitative version of~\cite[Lemma 4.3]{TW17},
we conclude $0 \geqslant \frac{C_{8}}{2}tr_{\widetilde{g}}g - C_{6}$
which implies the desired estimate

\noindent
\[
\eta_{1} \leqslant (n-1) ||f||_{L^{\infty}} \left( \frac{2 C_{6}}{C_{8}} \right)^{n-1}.
\]
\end{proof}

\begin{lemma} \label{lema: final bound C2 case 2}
We have, at  $x_{0}$,

\noindent
\begin{equation*}
\begin{split}
E  & := \frac{\sum_{i,j} \widetilde{g}^{i\bar{\imath}} \widetilde{g}^{j\bar{\jmath}} ( g_{j\bar{\imath}} \sum_a u_{a\bar{a} \bar{1}} - u_{j\bar{\imath} \bar{1}}+ \widehat{h}_{j\bar{\imath}\bar{1}}) (g_{i\bar{\jmath}} \sum_b u_{b\bar{b} 1} - u_{i\bar{\jmath}1} + \widehat{h}_{i\bar{\jmath}1})}{{(n-1)^2\eta_{1}}}   - \sum_{i = 2}^{n} \frac{\Theta^{i\bar{\imath}} | \eta_{1i}|^2}{(\eta_{1})^2} \\
& + \psi_{2}^{\prime \prime} \sum_i \Theta^{i\bar{\imath}} |u_i|^2 + \psi_{1}^{\prime \prime} \sum_{i = }^{n} \Theta^{i\bar{\imath}} \left| \sum_p u_p u_{\bar{p} i} + \sum_p u_{pi} u_{\bar{p}} \right|^2 \geqslant - \frac{C_{8}}{2} tr_{\widetilde{g}}g
\end{split}
\end{equation*}

\noindent
for $C_{8} = 2\frac{n}{n-1}|| h ||_{C^{1}}^{2}$.
\end{lemma}

\begin{proof}
Using~\eqref{eq: Qi}, $\psi_{1}^{\prime \prime} = 2(\psi_{1})^{2}$ and~\cite[Proposition 2.3]{HMW10} that is $\forall a;b \in \mathbb{C}^{n}$ and $ 0 < \delta^{\prime} < 1$, $| a + b |^{2} \geqslant \delta^{\prime} | a |^{2} - \frac{\delta^{\prime}}{1 - \delta^{\prime}} | b |^{2}$
we have

\noindent
\begin{equation} \label{ineq: lemma C2 first bound}
\begin{split}
\psi_{1}^{\prime \prime} \sum_{i = 2}^{n} \Theta^{i \bar{\imath}} \left| \sum_{p}u_{p} u_{\bar{p}i} + \sum_{p} u_{pi} u_{\bar{p}}  \right|^{2} \geqslant 2 \delta \sum_{i = 2}^{n} \Theta^{i \bar{\imath}} \frac{| \eta_{1 i} |^{2}}{(\eta_{1})^{2}} - \frac{2\delta (\psi_{2}^{\prime})^{2}}{1 - \delta} \sum_{i = 2}^{n} \Theta^{i \bar{\imath}} |u_{i}|^{2}
\end{split}    
\end{equation}
 
Combining it with the fact that $\psi_{2}^{\prime \prime} \geqslant \frac{2 \varepsilon}{1 - \varepsilon} (\psi_{2}^{\prime})^{2}$ for all $\varepsilon \leqslant \frac{1}{2A + 1}$, recall $\delta = \frac{1}{2A + 1}$, we obtain:

\noindent
\begin{equation} \label{ineq: main ineq lema C2}
\begin{split}
E \geqslant \frac{\sum_{i,j} \widetilde{g}^{i\bar{\imath}} \widetilde{g}^{j\bar{\jmath}} ( g_{j\bar{\imath}} \sum_a u_{a\bar{a} \bar{1}} - u_{j\bar{\imath} \bar{1}}+ \widehat{h}_{j\bar{\imath}\bar{1}}) (g_{i\bar{\jmath}} \sum_b u_{b\bar{b} 1} - u_{i\bar{\jmath}1} + \widehat{h}_{i\bar{\jmath}1})}{{(n-1)^2\eta_{1}}}   - (1 - 2\delta)\sum_{i = 2}^{n} \frac{\Theta^{i\bar{\imath}} | \eta_{1i}|^2}{(\eta_{1})^2}
\end{split}
\end{equation} 

Commuting derivatives, erasing unhelpful positive terms, using at $x_{0}$~\eqref{eq: tensor eta},
$g_{1 \bar{\imath}} = 0$ for $i >1$ and~\cite[Proposition 2.3]{HMW10} again we obtain:

\noindent
\begin{equation} \label{ineq: final bound 1 lema C2}
\begin{split}
& \frac{\sum_{i,j} \widetilde{g}^{i\bar{\imath}} \widetilde{g}^{j\bar{\jmath}} ( g_{j\bar{\imath}} \sum_a u_{a\bar{a} \bar{1}} - u_{j\bar{\imath} \bar{1}}+ \widehat{h}_{j\bar{\imath}\bar{1}}) (g_{i\bar{\jmath}} \sum_b u_{b\bar{b} 1} - u_{i\bar{\jmath}1} + \widehat{h}_{i\bar{\jmath}1})}{{(n-1)^2\eta_{1}}}  \\
& \geqslant (1 - \frac{\delta}{2})\sum_{i = 2}^{n} \frac{\widetilde{g}^{i\bar{\imath}} \widetilde{g}^{1 \bar{1}} | \eta_{1 \bar{\imath}} |^{2}}{(n-1)^{2} \eta_{1}} - \frac{1 - \frac{\delta}{2}}{\frac{\delta}{2}} \sum_{i = 2}^{n} \frac{\widetilde{g}^{i\bar{\imath}} \widetilde{g}^{1 \bar{1}} | [(tr_{g}h)g_{1\bar{1}}]_{\bar{\imath}} |^{2}}{(n-1)^{2} \eta_{1}}.
\end{split}
\end{equation}

In case 2 we have $\widetilde{g}^{i \bar{\imath}} \leqslant \frac{1}{(1 - \delta)\lambda_{n}}$ for $i >1$. Without loss of generality $(1 - \delta)\lambda_{n} \geqslant \frac{1 - \frac{\delta}{2}}{\frac{\delta}{2}}$ then

\noindent
\[
\frac{1 - \frac{\delta}{2}}{\frac{\delta}{2}} \sum_{i = 2}^{n} \frac{\widetilde{g}^{i \bar{\imath}} \widetilde{g}^{1\bar{1}} | [(tr_{g}h)g_{1\bar{1}}]_{\bar{\imath}} |^{2}}{(n-1)^{2}\eta_{1}} \leqslant \frac{C_{8}}{2} tr_{\widetilde{g}}g,
\]

\noindent
for $C_{8} = 2 \frac{n || h ||_{C^{1}}^{2}}{n-1}$, since the coordinates around $x_{0}$ are $g$-normal.
The only step left is to prove

\noindent
\[
(1 - \frac{\delta}{2})\sum_{i = 2}^{n} \frac{\widetilde{g}^{i\bar{\imath}} \widetilde{g}^{1 \bar{1}} | \eta_{1 \bar{\imath}} |^{2}}{(n-1)^{2} \eta_{1}} \geqslant (1 - 2\delta) \sum_{i = 2}^{n} \Theta^{i \bar{\imath}} \frac{| \eta_{1 \bar{i}} |^{2}}{(\eta_{1})^{2}},
\]

\noindent
which is exactly~\cite[Equation (4.39)]{TW17}
hence we refer the reader to it for the proof.
\end{proof}

\subsection{Proof of the \texorpdfstring{$C^{2}$}{C2} estimate on holomorphic families}

By Theorem~\ref{teo: C2 estimate hol fam}
we know that the constant for a $C^{2}$ norm depends on:

\begin{enumerate}
    \item $B>0$ uniform upper and lower bound on the bisectional curvature.
    
    \item $C_{L^{\infty}}$ the uniform $L^{\infty}$ bound on the solution to~\eqref{eq: n-1 MA HF}.
    
    \item The $C^{2}$ norm of $\log(f)$ and the metrics $\omega$ and $\beta$, also the $L^{\infty}$ norm of $f$.
    
\end{enumerate}

Under~\ref{set: sing family setting} without singular fibers
we have the uniformity of these constants:

\begin{enumerate}
    \item By compactness of the fibers the constant $B$ can be taken uniform on compact subsets of the base manifold $\mathbb{D}$, 
    hence it is uniform in the holomorphic family.

    \item By Theorem~\ref{teo: C0 Vincent} the constant $C_{L^{\infty}}$ is uniform even if the family has singularities.

    \item All the norms will be uniform on compact subsets of $\mathbb{D}$ by compactness as well.

\end{enumerate}

Hence, the $C^{2}$ bound, that depends on the $|\nabla u|^{2}_{g}$,
is uniform along a holomorphic family.

\section{Proof of~\ref{teo: estimates HF}}\label{CYG: C1 e HO HF}

\subsection{\texorpdfstring{$C^{1}$}{C1} Estimate on holomorphic families}\label{CYG: C1}

Since the $C^{2}$-estimate on the previous section depends on $\sup_{M}|\nabla u|_{g}^{2}$
we adapt the blow-up argument of~\cite{TW17}, based on~\cite{DK17}, to obtain a $C^{1}$ estimate for Holomorphic families.

\begin{theorem}\label{teo: C1 hol families}

Given a holomorphic family $(\mathcal{M}, \beta, \omega)$
in the fixed setting. The there is a constant $C>0$ which 
only depends on $||f||_{C^{2}(\mathcal{M},\omega)}$ and $(M, \beta_{0}, \omega_{0})$, 
such that

\noindent
\[
\underset{\mathbb{D}_{\frac{1}{2}}}{\sup} |\nabla u|_{\omega} \leqslant C.
\]

In particular, the constant $C$ does not depend on the complex structure.

\end{theorem}

\begin{proof}

Assume by contradiction, that $||f||_{C^{2}(\mathcal{M},\omega)} \leqslant C$ 
uniform and we have $ u_{t}, C_{t}$ such that:

\noindent
\[
\underset{M_{t}}{\sup} | \nabla u_{t}|_{\omega_{t}} = C_{t}
\]

\noindent
with $C_{t} \to +\infty$ as $t \to t_{0}$, for simplicity
we'll set $t_{0} = 0$. Then, we take a sequence of indexes $\{ t_{j}\}$,
simplifying notation just as $\{j\}$, such that $\underset{M_{j}}{\sup} |\nabla u_{j}|_{\omega_{j}} = C_{j}$,
for $C_{j} \to +\infty$ as $j \to +\infty$.

We have: $\underset{M_{j}}{\sup}~ u_{j} = 0$, $\underset{M_j}{\sup}~ |u_j| \leqslant C$ (independent of $j$ from the $L^{\infty}$-estimate) and

\noindent
\[
\alpha_{j} = \beta_{j} + \frac{1}{n-1}((\Delta_{j}u_{j})\omega_{j} - dd^{t}u_{j}), \quad
\alpha^{n}_{j} = f_{j} \omega_{j}^{n}.
\]

Let $x_j = \underset{M_{j}}{\text{arcmax}}~ |\nabla u_{j}|$. 
Up to extracting a subsequence $x_j \to x_{0} \in M = M_{0}$, for $t_{j} \to 0$.

Since the family is holomorphic then locally one has 
$V \subset \mathcal{M}$ open such that $V_{t} \subset X_{t}$,
$V_{0} \subset X = X_{0}$, $t \in I \subset B$ and set
$V = B_{2}(0) \times I$ such that $\omega(0) = \omega_{\mathbb{C}^{n}} = \sum_{j} i dz^{j} \wedge d \overline{z}^{j}$
for $\omega = \omega_{0}$, metric on fiber $X$ whenever confusion can be avoided.

Assume for $j$ large enough $x_{t_{j}}$ are all in $B_{1}(0) \times I^{\prime}$;
$I^{\prime} \Subset I$. Define in $B_{C_{j}}(0) \times I \subset \mathbb{C}^{n} \times I$:

\noindent
\[
\widehat{u}(z, t) = u_{j}(C^{-1}_{j} z + x_{j} ; t) ;  \quad \Phi_{j}: \mathbb{C}^{n} \to \mathbb{C}^{n}, \Phi_{j}(z) = C^{-1}_{j} z + x_{j}
\]

Since the family is holomorphically trivial, we will
have $u_{j},u$ well defined in small neighbourhoods and denote
for simplicity as if all the points are on $B_{1}(0) \subset X_{0}$. Then we have:

\noindent
\[
\underset{B_{C_{j}(0)}}{\sup}~ |\widehat{u}_{t,j}| \leqslant C; \forall t \in I, \forall j~~\text{big enough}, \quad \text{for}~~ \widehat{u}_{t,j} = \widehat{u}_{j}(\cdot , t).
\]

\noindent
and fixing $t \in I$ we have

\noindent
\[
\nabla^{g_{t}} \widehat{u}_{t,j} = \nabla^{g_{t}} u_{t,j} C^{-1}_{t,j} Id
\]

\noindent
then $| \nabla^{g_{t}}\widetilde{u}_{t,j} |(0) = 1$.
Hence $\underset{X_{j}}{\sup}~| \nabla^{g_{j}} \widehat{u}_{j} | \leqslant C$, independent of $t,j$.
By the $C^{2}$ estimate:

\noindent
\[
\underset{B_{C_{j}}(0)}{\sup}~| \partial \overline{\partial} \widehat{u}_{j} |_{\omega_{\mathbb{C}^{n}}} \leqslant C C_{j}^{-2} \underset{X_{j}}{\sup}~| \partial \overline{\partial}u_{j} |_{\omega} \leqslant C
\]

As in \cite{TW17}~by the Sobolev embedding and the elliptic estimates
for $\Delta$ for each given $K \subset \mathbb{C}^{n}$ compact, $0 < \gamma < 1$
and $p>1$, there is a constant $C$ such that:

\noindent
\[
|| \widehat{u}_{j} ||_{C^{1,\gamma}(K)} + || \widehat{u}_{j} ||_{W^{2,p}(K)} \leqslant C.
\]

Hence, there is a subsequence of $\widehat{u}_{j}$ that converges to $U$ in $\mathbb{C}^{n}$ with
$\underset{\mathbb{C}^{n}}{\sup}~(|U| + |\nabla U|) \leqslant C$ and $\nabla U (0) \neq 0$, consequently non-constant.
Finally, the rest of the proof follows verbatum the one in \cite{TW17}
That is, one proves that $U$ is maximal in $\mathbb{C}^{n}$, hence satisfy all the hypothesis in Theorem~\ref{teo: liouville thm n-1 psh},
which implies $U$ is constant, arriving at a contradiction.
\end{proof}

\subsection{Higher-Order Estimates on holomorphic families}\label{CYG: HO}

As for the classical Monge-Ampère equation the higher-order estimates are purely local.
Hence, one can addapt a complex Evans-Krylov type argument as in~\cite[Theorem 6.1]{TW17}
to obtain a $C^{2,\gamma}$-estimate, that depend on 2 derivatives of $\log(f)$ and use Schauder estimates with a bootstraping argument to conclude~\ref{teo: estimates HF}.
Moreover, one can use the Evans-Krylov-Caffarelli theory developed in~\cite[Theorems 1.1 and 1.2]{TWWY15} to obtain estimates that depend only on Hölder bounds of $\log(f)$, on the metrics and the $C^{2}$ estimate on the solution.
Implying explicitly that the estimates remain uniform on the holomorphic family by Sections~\ref{CYG: C2} and~\ref{CYG: C1}.

This concludes the proof of~\ref{teo: estimates HF}.

\bibliographystyle{alpha}
\bibliography{refs.bib}

\vspace{4pt}
\hrule

\end{document}